\documentclass[b5paper,10pt]{article}
\usepackage{times, url}
\usepackage{amsfonts}
\usepackage{eufrak}

\usepackage{amsmath,amssymb,amsthm,amsfonts}
\usepackage{graphicx}
\usepackage{color}
\usepackage{mathrsfs}
\usepackage[utf8]{inputenc}
\usepackage[nottoc,numbib]{tocbibind}

\usepackage{amsfonts}
\usepackage{amssymb}
\usepackage{amsmath}
\usepackage{amstext}
\usepackage{amsopn}
\usepackage{amsbsy}
\usepackage{amscd}
\usepackage{amsxtra}
\usepackage{esint}
\usepackage{amsthm}

\usepackage{epsfig}
\usepackage{pstricks,multido,pst-node}

\usepackage[misc]{ifsym}

\usepackage{bibspacing}
\newtheorem{theorem}{Theorem}[section]
\newtheorem{lemma}[theorem]{Lemma}
\newtheorem{remark}{Remark}
\newtheorem{corollary}[theorem]{Corollary}

\theoremstyle{theorem}

\newcommand{\ds}{\displaystyle}

\newcommand{\eps}{\epsilon}

\newcommand{\dxx}{\,\mathrm{d}{x}}
\newcommand{\dtt}{\,\mathrm{d}{t}}

\newcommand{\dsx}{\,\mathrm{d}\sigma_{x}}
\newcommand{\mta}{\pmb{\alpha}}
\newcommand{\mtb}{\pmb{\beta}}
\newcommand{\mte}{\pmb{\eta}}
\newcommand{\mtu}{\mathtt{u}}
\newcommand{\bigu}{\pmb{\mathrm{U}}}
\newcommand{\uss}{(\doo^2{\ueps})}
\newcommand{\us}{(\doo{\ueps})}
\newcommand{\alphas}{(\doo{\aeps})}
\newcommand{\betas}{(\doo{\beps})}
\newcommand{\ueps}{{u_{\eps}}}
\newcommand{\aeps}{{\alpha_{\eps}}}
\newcommand{\beps}{{\beta_{\eps}}}
\newcommand{\geps}{{\gamma_{\eps}}}
\newcommand{\doo}{\boldsymbol{\mathrm{D}}}
\newcommand{\dru}{\delta_{R(\bigu-\theta_0)}}
\newcommand{\drup}{\delta_{R\bigu'^2}}

\numberwithin{equation}{section}

\allowdisplaybreaks

\usepackage{color,hyperref}
   \catcode`\_=11\relax
    \newcommand\email[1]{\_email #1\q_nil}
    \def\_email#1@#2\q_nil{%
      \href{mailto:#1@#2}{{\emailfont #1\emailampersat #2}}
    }
    \newcommand\emailfont{\sffamily}
    \newcommand\emailampersat{{\color{cyan}\scriptsize@}}
    \catcode`\_=8\relax 
		
		\usepackage{hyperref}
\hypersetup{
    colorlinks=true,
    linkcolor=blue,          
    citecolor=blue,        
    filecolor=magenta,      
    urlcolor=cyan           
}
 
\begin{document}
\setcounter{page}{1} 
\noindent
{\scriptsize Under consideration for publication in {\bf\slshape Annali di Matematica Pura ed Applicata}}\vspace{-3.5mm} \\
\rule{39em}{0.2pt}\\
{\large\bf{Nontrivial boundary structure in a Neumann problem on
 balls with radii\vspace{1mm}\\  tending to infinity}}
\vspace{1mm}\\
{Chiun-Chang Lee\footnote{\scriptsize Institute for Computational and Modeling Science, National Tsing Hua University, Hsinchu 30014, Taiwan~ (\,\Letter~\email{chlee@mail.nd.nthu.edu.tw}\!).}}\footnote{\scriptsize During writing of the manuscript, the author was visiting the Nihon University and the Saitama University. It is pleasure to thank Professors Masashi Mizuno and Shingo Takeuchi for the kind hospitality and for helpful conversations towards the improvement of the earlier version. Financial support of this work was also partially provided by the Ministry of Science and Technology of Taiwan (Grant ID: MOST-107-2115-M-007-004).} 

\noindent
\footnotesize


\begin{abstract}\footnotesize
This note introduces a class of nonlinear Neumann problems on balls expanding with the radii tending towards infinity. Performing singular perturbation arguments, we establish the corresponding concentration phenomenon and refined asymptotic expansions with the precise first two order terms. In doing so, we obtain the nontrivial boundary structure of solutions with effects coming from the nonlinear Neumann boundary condition and the boundary mean curvature varied with expanding domains.\vspace{5pt}
\end{abstract}

\tableofcontents


\section{Introduction}
\noindent

This work is motivated by some stationary reaction--diffusion models and electro-chemistry models in a reactor of \textbf{macroscopic length scale} involving nonlinear adsorption process on the surface~\cite{BS1973,DS2018,MN2011,Shibata2006,T2001}. In such a situation, the region for a chemical substance to diffuse across is much larger compared with a reaction process~\cite{BY1999,CW1999,Z2012}.

Mathematically, one considers the related differential equations with nonlinear Neumann boundary conditions in expanding domains, where the nonlinear source describes the absorption process, and the boundary effect is associated with the adsorption process; see, e.g., \cite{S1993}. Here the expanding domain means that the diameter of a large domain keeps increasing towards infinity. Such expanding domains may formally approach the entire space, the half space or an unbounded exterior domain. However, due to the nonlinear boundary effect, the asymptotic behavior of solutions varied with the expanding domain are totally different from the entire solutions. Since the domain keeps getting large, let us imagine in mind firstly that as the domain boundary expands out with the same distance along the outward normal direction, the corresponding solutions asymptotically vary with the expanding domain, and its asymptotics remains to be strongly affected by nonlinear boundary conditions~\cite{A1976,C1996}. Essentially, such a phenomenon can be investigated under appropriate scales related to the diameter of the domain. Accordingly, the problem is equivalently transformed into singularly perturbed equations in finite domains. For the large domain with diameter tending to infinity, an important issue arises about the optimal upper bounds and the asymptotic behavior of solutions with respect to the domain geometry.  

To basically understand the influence of expanding domains on solutions, we focus on the domain $B_R$ a ball of large radius $R\gg1$ centered at the origin in $\mathbb{R}^N$, $N\geq2$. We shall investigate a class of semilinear elliptic equations which are more general than models in \cite{S1993}. The model reads
\newpage
\begin{align}
\nabla\cdot(\mta(|x|)\nabla\mtu(x))=\mtb(|x|)f(\mtu(x))\,\,\,\,\mbox{in}&\,\quad\,B_R,\label{eq1}\\
\frac{\partial\mtu}{\partial{\vec{\nu}}}(x)=\mte(\mtu(x))\quad\quad\quad\mbox{on}&\,\,\,\,\partial{B}_R,\label{bd1}
\end{align}
where $\nabla$ and $\nabla\cdot$ are the gradient and the divergence operators, respectively. $|x|$ denotes the standard $N$-dimensional Euclidean norm, $\vec{\nu}=\vec{\nu}(x)$ is the unit outward normal vector to $\partial{B_R}$ at $x$, $\frac{\partial}{\partial{\vec{\nu}}}$ is the unit outward normal derivative, and functions $f$ and $\mte$  admit the following assumptions:
\begin{itemize}
\item[]
\begin{itemize}
\item[\textbf{(A1).}] $f\in\mbox{C}_{\mathrm{loc}}^{1,\tau}(\mathbb{R})$ with $\tau\in(0,1)$, $\ds\inf_{\mathbb{R}}f'>0$ and $f(\theta_0)=0$ for some $\theta_0\in\mathbb{R}$.
\item[\textbf{(A2).}] $\mte\in\mbox{C}_{\mathrm{loc}}^{1,\tau}(\mathbb{R})$ is monotonically decreasing and strictly positive in $\mathbb{R}$.
\end{itemize}
\end{itemize}

Equation (\ref{eq1}) has many practical applications in the fields of physics, chemistry and biology, where $\mta$ characterizes the diffusion, $\mtb$ is regarded as a spatially inhomogeneous reaction term for the absorption $f$, and $\mte$ admitting (A2) models a degradation process in $B_R$ which is compensated by adsorption through $\partial{B_R}$. For a simplified case $\mta\equiv1$ and $\mtb\equiv1$, we refer the reader to \cite[(2a) and (2b)]{S1993} for a typical model obeying assumptions (A1) and (A2). In this work, $\mta$ and $\mtb$ are treated in more general settings as follows:
\begin{itemize}
\item[]
\begin{itemize}
\item[\textbf{(A3).}] $\mta\in\mbox{C}_{\mathrm{loc}}^{2,\tau}([0,\infty))$ and $\mtb\in\mbox{C}_{\mathrm{loc}}^{1,\tau}([0,\infty))$ are bounded above and have positive infima, and
\begin{center}
 $\mtb(r)r^{N-1}$ is increasing to $r>0$.
\end{center}
 Moreover, for $\mta_R(r):=\mta(r)\chi_{[0,R]}(r)$ and $\mtb_R(r):=\mtb(r)\chi_{[0,R]}(r)$ restricted in the domain $[0,R]$ with sufficiently large $R$, there exists $k^*\in(0,1)$ independent of $R$ such that
\begin{align}\label{id0207-add}
\lim_{R\to\infty}\sup_{r\in[k^*R,R)}\Big({R}\left({\left|\mta_R'(r)\right|}+{\left|\mtb_R'(r)\right|}\right)+R^{2}{|\mta_R''(r)|}\Big)\in(0,\infty).
\end{align}
\end{itemize}
\end{itemize}

As an example in (A3), we introduce a smooth function $\mta_R=\mta\chi_{[0,R]}$ satisfying property (\ref{id0207-add}) with $\mta(r)=k^*$ for $r\in[0,k^*R]$, $\mta(r)\in[k^*,1]$ for $r\in[k^*R,kR]$, and $\mta(r)=1$ for $r\in[kR,\infty)$, where $k^*\in(0,1)$ and $k>1$ are constants independent of $R$. 


For (\ref{eq1}), one naturally considers the boundary condition $\mta(|x|)\frac{\partial\mtu}{\partial{\vec{\nu}}}(x)=\mte(\mtu(x))$. Here we use (\ref{bd1}) since $\mta$ is a positive constant on $\partial{B}_R$. In the related issues, some previous works have been traced back to \cite{C1996,S1993}. Let us mention \cite{C1996,S1993}, where the optimal bounds for solutions of (\ref{eq1})--(\ref{bd1}) with $\mta\equiv1$ and $\mtb\equiv1$ have been investigated. However, at the best of our knowledge only partial results for the structure of solutions have been obtained. One of main difficulties lies on unknown boundary behavior of $\mtu$ and $\frac{\partial\mtu}{\partial\vec{\nu}}$ which interact with each other in the nonlinear boundary condition~(\ref{bd1}). 

Starting with an interior estimate, we prove that for any $R_0\in(0,R)$, 
\begin{align}\label{0606-m}
\max_{B_{R_0}}\left(\left|\mtu(x)-\theta_0\right|+\left(\frac{|x|}{R}\right)^{N-1}\left|\nabla\mtu(x)\right|\right)\leq\mathtt{L}_0e^{-{\mathtt{M}_0}(R-R_0)},
\end{align}
where $\mathtt{L}_0$ and $\mathtt{M}_0$ are positive constants independent of $R$ and $R_0$ (cf. \eqref{id0207-1}). As a consequence, $\mtu$ behaves as a flat core (converges to $\theta_0$ exponentially) in any compact subset $K$ of $B_R$ as $\mathrm{dist}(\partial{K},\partial{B_R})\stackrel{R\to\infty}{-\!\!\!-\!\!\!-\!\!\!\longrightarrow}\infty$. Since $\theta_0$ does not satisfy the boundary condition~\eqref{bd1}, $\mtu$ is non-trivial near the boundary. To deal with the boundary asymptotics, one can observe that under the scale $x=R\widetilde{x}$, (\ref{eq1}) becomes a singularly perturbed model in the domain $B_1:=\{\widetilde{x}\in\mathbb{R}^N:|\widetilde{x}|<1\}$ with a parameter $\frac{1}{R^2}\to0$, and on the boundary $\partial{B_1}$ the outward normal derivative in (\ref{bd1}) has a parameter $\frac{1}{R}\to0$ (see, e.g., (\ref{intro-1}) and the equation (\ref{eq3})--(\ref{bd3})). Hence, the singularity of $|\nabla\mtu|$ near $\partial{B_R}$ introduces additional difficulties when trying to implement the standard technique of matching asymptotic expansions that do work for singularly perturbed semilinear elliptic problems. In this work, we are devoted to refined boundary asymptotics of $\mtu$ as $R\gg1$. We propose a new analysis technique  based on arguments in \cite{GMW2016,l2019,Shibata2003,Shibata2003-ann,Shibata2004,T2000} and \cite[Proposition 2]{l2016}. For the fist situation, we assume that the perturbation of $\frac{\mtb(R)}{\mta(R)}-\mu_0$ with respect to $R\gg1$ is sufficiently small in the sense 
\begin{align}\label{add0605}
\lim_{R\to\infty}R\left(\frac{\mtb(R)}{\mta(R)}-\mu_0\right)=0,
\end{align}
where $\mu_0$ is a positive constant independent of $R$. Then the boundary asymptotic expansions at each boundary point $x_{\mathrm{bd}}\in\partial{B}_R$ can be formally depicted as follows (see (\ref{id0207-2})--(\ref{mathcal-h-0215}) for the rigorous versions):
\begin{align}
\mtu(x_{\mathrm{bd}})&\,\boldsymbol{\stackrel{R\gg1}{\approx}}\quad\,\,\,{p_0}+\frac{\ds\int_{\theta_0}^{p_0}\sqrt{\frac{F(t)-F(\theta_0)}{F(p_0)-F(\theta_0)}}\,\mathrm{d}t}{\ds\mu_0\frac{f(p_0)}{\mte(p_0)}-{\mte'(p_0)}}\left(\frac{N-1}{R}+\frac{\mta'(R)}{2\mta(R)}+\frac{\mtb'(R)}{2\mtb(R)}\right),\label{0224-intro-1}\\[-0.7em]
&\notag\\[-0.7em]
\frac{\partial\mtu}{\partial\vec{\nu}}(x_{\mathrm{bd}})
&\,\boldsymbol{\stackrel{R\gg1}{\approx}}\,\mte(p_0)+\frac{\ds\mte'(p_0)\int_{\theta_0}^{p_0}\sqrt{\frac{F(t)-F(\theta_0)}{F(p_0)-F(\theta_0)}}\,\mathrm{d}t}{\ds\mu_0\frac{f(p_0)}{\mte(p_0)}-{\mte'(p_0)}}\left(\frac{N-1}{R}+\frac{\mta'(R)}{2\mta(R)}+\frac{\mtb'(R)}{2\mtb(R)}\right),\label{0224-intro-2}
\end{align}
where $\mathtt{a}\boldsymbol{\stackrel{R\gg1}{\approx}}\mathtt{b}$ means $R(\mathtt{a}-\mathtt{b})\to0$ as $R\to\infty$, and
\begin{align}\label{big-f}
F(t)=\int_0^tf(s)\,\mathrm{d}s
\end{align}
is the primitive of $f$, and $p_0>\theta_0$ is uniquely determined by $\mte(p_0)=\sqrt{{2}{\mu_0}(F(p_0)-F(\theta_0))}$ (cf. (\ref{equ-p})). It is clear that even if $R$ is large, $\mtu$ is strongly influenced by the nonlinear effect of (\ref{bd1}) on the boundary. We stress that the asymptotics (\ref{0224-intro-1}) and (\ref{0224-intro-2}) are obtained under assumption~\eqref{add0605}, i.e., $\frac{\mtb(R)}{\mta(R)}\boldsymbol{\stackrel{R\gg1}{\approx}}\mu_0$. In light of (\ref{0224-intro-1}) and (\ref{0224-intro-2}), solutions asymptotically expand as the radius of the domain $B_R$ tends to infinity, and $\mta$, $\mta'$, $\mtb$, $\mtb'$, $\mte$, $\mte'$ and the curvature $\frac{1}{R}$ have significant influence on the structure of solutions. Note also that even if $|x_{\mathrm{bd}}|=R\to\infty$, both $\mtu(x_{\mathrm{bd}})$ and $\frac{\partial\mtu}{\partial\vec{\nu}}(x_{\mathrm{bd}})$ remain finite and positive. Hence, $\mtu$ forms a boundary layer with the concentration phenomenon near the boundary $\partial{B}_R$. The rigorous boundary asymptotic expansions of~$\mtu$ and $\frac{\partial\mtu}{\partial\vec{\nu}}$ will be presented in Theorem~\ref{thm1}. For an application of such asymptotics, we refer the reader to Corollary~\ref{rk1}.  To describe the related boundary concentration phenomena of the solution~$\mtu$ via a theoretical perspective, we show that $R(\mtu(x)-\theta_0)$ and $R|\nabla\mtu(x)|^2$ weakly converge to Dirac measures concentrating at infinity as $R$ tends towards infinity. Such phenomena will be described in Theorem~\ref{thm2}.

Despite the crucial roles of $\mu_0$ and $p_0$ in asymptotics \eqref{0224-intro-1} and \eqref{0224-intro-2}, assumption~(\ref{add0605}) implies that the perturbation of $\frac{\mtb(R)}{\mta(R)}$ with respect to $\mu_0$ is actually rather small than the curvature of $\partial{B}_R$ as $R$ is sufficiently large. To study further the influence of small perturbation of $\frac{\mtb(R)}{\mta(R)}-\mu_0$ on asymptotic expansions of $\mtu(x_{\mathrm{bd}})$ and $\frac{\partial\mtu}{\partial\vec{\nu}}(x_{\mathrm{bd}})$, we shall consider the situation $\ds\liminf_{R\to\infty}$~$R\Big|\frac{\mtb(R)}{\mta(R)}-\mu_0\Big|>0$ instead of (\ref{add0605}). In the final Section~\ref{sec-ap} we will establish the corresponding boundary asymptotic expansions in Corollary~\ref{cor0603} which are more complicated than \eqref{0224-intro-1} and \eqref{0224-intro-2}. As an application of Corollary~\ref{cor0603}, we focus particularly on the case
\begin{align}\label{0608-hap}
\lim_{R\to\infty}\frac{\mtb(R)}{\mta(R)}=\mu_0\,\,\,\mathrm{and}\,\,\,
\lim_{R\to\infty}R^{\tau_*}\left|\frac{\mtb(R)}{\mta(R)}-\mu_0\right|\in(0,\infty)\,\,\mathrm{for\,\,some}\,\,\tau_*>0.
\end{align}
For doing so, the effects of boundary curvature $\frac{1}{R}$ and the perturbation of $\frac{\mtb(R)}{\mta(R)}-\mu_0$ on boundary asymptotics of $\mtu$ and $\frac{\partial\mtu}{\partial\vec{\nu}}$ will be classified via three situations $\tau_*\in(0,1)$, $\tau_*=1$ and $\tau_*\in(1,\infty)$. Such a result can be found in Remark~\ref{rk0621}.

\section{Statement of the main results}
\noindent

The associated energy functional of \eqref{eq1}--(\ref{bd1}) is defined by
\begin{align*}
\mathcal{E}[\mathtt{v}]=\int_{B_R}\frac{\mta(|x|)}{2}|\nabla\mathtt{v}|^2+\mtb(|x|)F(\mathtt{v})\dxx\,-\mta(R)\int_{\partial{B_R}}\int_{\theta_0}^{\mathtt{v}}\mte(t)\dtt\dsx,\,\,\mathtt{v}\in\mathrm{H}^1(B_R).
\end{align*}
Let us fix $R>0$. Since $\ds\min_{\mathbb{R}}{F}={F}(\theta_0)$ (by (A1)), together with (A2)--(A3) we verify that $\mathcal{E}$ is bounded below over $\mathrm{H}^1(B_R)$. Thus, applying the standard direct method to $\mathcal{E}$, one immediately obtains the existence of weak solutions to \eqref{eq1}--\eqref{bd1}.  Thanks again to (A1)--(A3), for each fixed $R>0$ we can further follow the standard argument consisting of the maximum principle and the elliptic regularity theorem (cf. \cite{LSU1968}) to show that (\ref{eq1})--(\ref{bd1}) has a unique solution $\mtu\in\mbox{C}^1(\overline{B_R})\cap\mbox{C}^{\infty}(B_R)$ satisfying $\mtu(x)\geq\theta_0$, $\forall{x}\in\overline{B_R}$. 
In particular, the uniqueness implies that
$\mtu(x)=\bigu(|x|)$ is radially symmetric in $B_R$, where $\bigu$ is the unique solution of 
\begin{align}
(r^{N-1}\mta(r)\bigu'(r))'=&r^{N-1}\mtb(r)f(\bigu(r)),\,\,r\in(0,R),\label{eq2}\\
\bigu'(0)=&0,\,\,\bigu'(R)=\mte(\bigu(R)),\label{bd2}
\end{align}
and satisfies
\begin{align}\label{id-mtu}
\bigu(r)\geq\theta_0\,\,\mbox{in}\,\,[0,R]. 
\end{align}
Along with (A1) yields $f(\bigu(r))\geq0$ in $[0,R]$. Notice also that $\mta(r)$ and $\mtb(r)$ are positive in $(0,R)$. Since $\bigu$ solves (\ref{eq2}) and satisfies $\bigu'(0)=0$, we know that $r^{N-1}\mta(r)\bigu'(r)$ is increasing to $r$ and, consequently,
\begin{align}\label{ali88}
\bigu'(r)\geq0\,\,\mbox{in}\,\,[0,R].
\end{align}
Accordingly, $\mtu$ is monotonically increasing in the sense that $\mtu(x)\geq\mtu(y)$ if $|x|\geq|y|$. It should also be mentioned that $\mtu$ is stable since the second variation of $\mathcal{E}[\mtu]$ with respect to compactly supported smooth perturbations $\xi$ is non-negative, i.e.,
\begin{align*}
Q_{\mtu}[\xi]:=\int_{B_R}{\mta(|x|)}|\nabla\xi|^2+\mtb(|x|)f'(\mtu)\xi^2\dxx-\mta(R)\int_{\partial{B_R}}\mte'(\mtu)\xi^2\dsx\geq0,\,\,\forall\,\xi\in\mathrm{C}_c^1({B_R})
\end{align*}   
 (trivially due to (A1)--(A3)).

\subsection{Boundary structure and concentration phenomena}
\noindent

The main goal of this work is to establish asymptotic behavior of solution $\bigu$ as $R$ goes to infinity. Later on we will prove that both $\bigu$ and $\bigu'$ are uniformly bounded in $[0,R]$ for all $R>0$. To establish the refined asymptotics, asymptotic expansions of $\mta(R)$ and $\mtb(R)$ with respect to $R\gg1$ are required. In what follows we continue along the relation~(\ref{id0207-add}) to further assume that as $R\to\infty$, $\frac{\mtb(R)}{\mta(R)}$ approaches a positive constant $\mu_0$ in the sense described in (\ref{add0605}), i.e.,
\begin{align}\label{mta-b-0211}
\frac{\mtb(R)}{\mta(R)}=\mu_0+\frac{o(1)}{R},\,\,\mbox{as}\,\,R\gg1,
\end{align}
 where $o(1)$ denotes the quantity approaching zero as $R$ goes to infinity. The first result is about an interior estimate of $\bigu$ and $\bigu'$ and refined, precise asymptotics for $\bigu(R)$ and $\bigu'(R)$. Particularly, the boundary asymptotic expansions involve the domain geometry and the behavior of $\mta'(R)$ and $\mtb'(R)$.

\begin{theorem}[Interior and boundary asymptotics]\label{thm1}
Assume (A1)--(A3). For $N\geq2$ and $R>0$, let $\bigu\in\mathrm{C}^1((0,R])\cap\mathrm{C}^{\infty}((0,R))$ be the unique solution of (\ref{eq2})--(\ref{bd2}). Then $\bigu$ is monotonically increasing in $[0,R]$. As $R\gg1$, $\bigu$ is strictly convex near the boundary, and there exist positive constants ${\mathtt{L}_0}$ and ${\mathtt{M}_0}$ independent of $R$ such that for $r\in[0,R]$,
\begin{align}\label{id0207-1}
|\bigu(r)-\theta_0|+\left(\frac{r}{R}\right)^{N-1}|\bigu'(r)|\leq{{\mathtt{L}_0}}e^{-{\mathtt{M}_0}(R-r)}.
\end{align}
Moreover,  if (\ref{mta-b-0211}) is satisfied, then the boundary asymptotics of $\bigu(R)$ and $\bigu'(R)$ involving the effects of $\mta'(R)$, $\mtb'(R)$ and the curvature $\frac{1}{R}$ are depicted as
\begin{align}
\bigu(R)=&\,\quad{p_0}+\boldsymbol{\mathtt{C}_{0}}\pmb{\mathcal{H}}(R)+\frac{o(1)}{R},\label{id0207-2}\\[-0.7em]
\notag
\\[-0.7em]
\bigu'(R)=&\,\mte(p_0)+\mte'(p_0)\boldsymbol{\mathtt{C}_{0}}\pmb{\mathcal{H}}(R)+\frac{o(1)}{R},\label{id0207-3}
\end{align}
where
\begin{align}\label{mathcal-h-0215}
\begin{cases}
\hspace*{10pt}\mathtt{C}_{0}&=\,\displaystyle\left(\mu_0\frac{f(p_0)}{\mte(p_0)}-{\mte'(p_0)}\right)^{-1}{\int_{\theta_0}^{p_0}\sqrt{\frac{F(t)-F(\theta_0)}{F(p_0)-F(\theta_0)}}\,\mathrm{d}t},\\[-0.2em]
\\[-0.2em]
\pmb{\mathcal{H}}(R)&=\,\displaystyle\frac{N-1}{R}+\frac{1}{2}\left(\frac{\mta'(R)}{\mta(R)}+\frac{\mtb'(R)}{\mtb(R)}\right).
\end{cases}
\end{align}
Here $p_0>\theta_0$ is uniquely determined by the nonlinear algebraic equation
\begin{align}\label{equ-p}
\mte(p_0)=\sqrt{{2}{\mu_0}(F(p_0)-F(\theta_0))},
\end{align}
and $F$ is defined in (\ref{big-f}).
\end{theorem}

Note that $\mathtt{C}_{0}$ is a positive coefficient  independent of $R$ (cf. (A1) and (A2)).
The uniqueness of equation (\ref{equ-p}) is trivially due to the fact that  $\mte$ is a decreasing function and $F$ is strictly increasing in $(\theta_0,\infty)$ (by (A1) and (A2)).

(\ref{id0207-2}) and (\ref{id0207-3}) provide fruitful information for the effects of $\mta$ and $\mtb$ on boundary asymptotics of $\bigu$. It should be mentioned a case
\begin{align*}
\frac{N-1}{R}+\frac{1}{2}\left(\frac{\mta'(R)}{\mta(R)}+\frac{\mtb'(R)}{\mtb(R)}\right)=\frac{o(1)}{R}\,\,\mathrm{as}\,\,R\gg1;
		\end{align*}
for example, $\mta(r)=\frac{N-1}{R}(R-r)+1$ and $\mtb(r)=\mu_0\mta(r)$ for $r\in[0,R]$. Then we have
\begin{equation*}
\bigu(R)=p_0+\frac{o(1)}{R}\,\,\text{and}\,\, \bigu'(R)=\mte(p_0)+\frac{o(1)}{R},
\end{equation*}
and conclude that the effect of the domain size on solution $\bigu$ is inconspicuous. Let us consider another special case where $\mta(r)\mtb(r)$ is a constant value as $r\geq{r}_0$ for some $r_0>0$. Then, as $R\gg1$, \eqref{mathcal-h-0215} implies $\pmb{\mathcal{H}}(R)=\frac{N-1}{R}$. In this case, $\bigu(R)$ and $\bigu'(R)$ are indeed varied with the boundary curvature, but the effect of $\mta$ and $\mtb$ on $\bigu(R)$ and $\bigu'(R)$ are quite slight. 

We shall also stress the importance of second order terms of (\ref{id0207-2}) and (\ref{id0207-3}). Note that $\ds\max_{[0,R]}\bigu=\bigu(R)\sim{p}$ and $\bigu'(R)\sim\mte(p_0)$ as $R\gg1$. When $\mte'(p_0)<0$ (cf. (A2)), by the second order terms of (\ref{id0207-2}) and (\ref{id0207-3}) one further gets
\begin{center}
$\displaystyle\frac{N-1}{R}+\frac{1}{2}\left(\frac{\mta'(R)}{\mta(R)}+\frac{\mtb'(R)}{\mtb(R)}\right)>0$ as $R\gg1$ $\pmb{\pmb{\Longleftrightarrow}}$ $\bigu(R)>p_0$ and $\bigu'(R)<\mte(p_0)$ as $R\gg1$.
\end{center}
In particular, if $\mta(r)=\alpha_1$ and $\mtb(r)=\beta_1$ are constants as $r$ is close to $R$, then for sufficiently large $R$, $\pmb{\mathcal{H}}(R)=\frac{N-1}{R}$, and $\bigu(R)>p_0$ and $0<\bigu'(R)<\mte(p_0)$. Moreover, some monotone properties for boundary asymptotics of $\bigu(R)$ and $\bigu'(R)$ with respect to $\mta'(R)$, $\mtb'(R)$ and the sufficiently large radius $R$ of the domain $B_R$ are stated as follows. 

\begin{corollary}\label{rk1}
Under the same hypotheses as in Theorem~\ref{thm1}, let $\mta_i\in\mbox{C}_{\mathrm{loc}}^{2,\tau}([0,\infty))$ and $\mtb_i\in\mbox{C}_{\mathrm{loc}}^{1,\tau}([0,\infty))$ satisfy (A3). Then we have
\begin{itemize}
\item[$\mathrm{(I).}$]  Let $\bigu_{\mta_i,\mtb_i}$ be the unique solution of (\ref{eq2})--(\ref{bd2}) with $(R,\mta,\mtb)=(R_i,\mta_i,\mtb_i)$, $i=1,2$, where $1<{R}_1<R_2$ and $\ds\sup_{R_1\gg1}$$\frac{R_2}{R_1}<\infty$. If $\frac{\mtb_i(R_i)}{\mta_i(R_i)}$ satisfies (\ref{mta-b-0211}) and
\begin{equation*}
\left(\frac{\mta_1'(R_1)}{\mta_1(R_1)}+\frac{\mtb_1'(R_1)}{\mtb_1(R_1)}\right)-\left(\frac{\mta_2'(R_2)}{\mta_2(R_2)}+\frac{\mtb_2'(R_2)}{\mtb_2(R_2)}\right)=\frac{o(1)}{R_1}. 
\end{equation*}
Then as $R_1$ is sufficiently large, there hold
\begin{align*}
 \bigu_{\mta_1,\mtb_1}(R_1)>\bigu_{\mta_2,\mtb_2}(R_2)>\theta_0 \quad\mathrm{and}\quad 0<\bigu_{\mta_1,\mtb_1}'(R_1)\leq\bigu_{\mta_2,\mtb_2}'(R_2).
\end{align*}
Moreover, when $\mte'(p_0)<0$, we have $0<\bigu_{\mta_1,\mtb_1}'(R_1)<\bigu_{\mta_2,\mtb_2}'(R_2)$ as $1\ll{R}_1<R_2$.
\item[$\mathrm{(II).}$] Let $\widetilde{\bigu}_{\mta_i,\mtb_i}$ be the unique solution of (\ref{eq2})--(\ref{bd2}) in $(0,R)$ with $(\mta,\mtb)=(\mta_i,\mtb_i)$, $i=1,2$. Assume further that 
\begin{align*}
\frac{\mtb_1(R)}{\mta_1(R)}\,\,and\,\,\frac{\mtb_2(R)}{\mta_2(R)}\,\,are\,\,positive\,\,constants\,\,independent\,\,of\,\,R, 
\end{align*}
and one of the following assumptions holds:
\begin{itemize}
 \item[$\mathrm{(i).}$] $\ds\frac{\mtb_1(R)}{\mta_1(R)}<\frac{\mtb_2(R)}{\mta_2(R)}$;
 \item[$\mathrm{(ii).}$] $\ds\frac{\mtb_1(R)}{\mta_1(R)}=\frac{\mtb_2(R)}{\mta_2(R)}$, $\ds\frac{\mta_1'(R)}{\mta_1(R)}+\frac{\mtb_1'(R)}{\mtb_1(R)}>\frac{\mta_2'(R)}{\mta_2(R)}+\frac{\mtb_2'(R)}{\mtb_2(R)}$ and $\mte'(p_0)<0$, 
\end{itemize}
then $\widetilde{\bigu}_{\mta_1,\mtb_1}(R)>\widetilde{\bigu}_{\mta_2,\mtb_2}(R)>\theta_0$ and $0<\widetilde{\bigu}_{\mta_1,\mtb_1}'(R)<\widetilde{\bigu}_{\mta_2,\mtb_2}'(R)$ as $R\gg1$. 
\end{itemize}
\end{corollary}

A discussion on Corollary~\ref{rk1}(II) is stated as follows:

\begin{remark}\label{rk-0311}
It seems that the standard comparison is difficult to imply Corollary~\ref{rk1}(II). 
Let us consider another situation that $\mta_i$ and $\mtb_i$ satisfy
\begin{align}\label{eq-0311}
\frac{\mtb_1(r)}{\mta_1(r)}\leq\frac{\mtb_2(r)}{\mta_2(r)}\,\,\,and\,\,\,\frac{\mta_1'(r)}{\mta_1(r)}\geq\frac{\mta_2'(r)}{\mta_2(r)},\,\,\forall\,r\in[0,R].   
\end{align}
Then, applying the standard PDE comparison to (\ref{eq2})--(\ref{bd2}) and using (\ref{id-mtu})--(\ref{ali88}), one obtains $\widetilde{\bigu}_{\mta_1,\mtb_1}\geq\widetilde{\bigu}_{\mta_2,\mtb_2}\geq\theta_0$ in $[0,R]$. In particular,  if $\widetilde{\bigu}_{\mta_1,\mtb_1}\not=\widetilde{\bigu}_{\mta_2,\mtb_2}$ at an interior point, then $\widetilde{\bigu}_{\mta_1,\mtb_1}(R)>\widetilde{\bigu}_{\mta_2,\mtb_2}(R)>\theta_0$. This is the same as the corresponding result in Corollary~\ref{rk1}(II), but the  conditions (i) and (ii) are far weaker than condition~(\ref{eq-0311}).
\end{remark}

Let us return to Theorem~\ref{thm1} which establishes refined asymptotics of $\bigu(R)$ and $\bigu'(R)$ under a strong assumption (\ref{mta-b-0211}). It should be stressed that if $\frac{\mtb(R)}{\mta(R)}\to\mu_0$ but it does not satisfy (\ref{mta-b-0211}), then the effect of the perturbation of $\frac{\mtb(R)}{\mta(R)}-\mu_0$ cannot be ignored. We will establish asymptotics of $\bigu(R)$ and $\bigu'(R)$ involving the effect of the perturbation of $\frac{\mtb(R)}{\mta(R)}-\mu_0$  in Section~\ref{sec-ap}; see \eqref{id0304-2}--\eqref{id0304-3}.

To see the concentration phenomenon of $\bigu$ near the boundary $r=R$ as $R\to\infty$, let us introduce a Dirac measure $\delta^{\infty}$ defined in the interval of non-negative extended real numbers, which satisfies $\delta^{\infty}(r)=0$ for $r\in(0,\infty)$ and $\int_0^{\infty}\delta^{\infty}(r)\,\mathrm{d}r=1$. We focus on the behavior of $\bigu$ in the region $(k^*R,R)$ and define
\begin{align}
\dru(r)=
\begin{cases}
R(\bigu(r)-\theta_0),\,\,&\mbox{for}\,\,r\in(k^*R,R),\\
0,\,\,&\mbox{for}\,\,r\in[0,k^*R]\cup[R,\infty),
\end{cases}
\end{align}
and
\begin{align}
\drup(r)=
\begin{cases}
R\bigu'^2(r),\,\,&\mbox{for}\,\,r\in(k^*R,R),\\
0,\,\,&\mbox{for}\,\,r\in[0,k^*R]\cup[R,\infty),
\end{cases}
\end{align}
where $k^*$ is defined in (A3). The following theorem confirms that $\dru$ and $\drup$ behave as Dirac measures at infinity in the following weak sense: 
\begin{align}
\dru\stackrel{R\to\infty}{-\!\!\!-\!\!\!-\!\!\!-\!\!\!-\!\!\!\rightharpoonup}&\left(\frac{1}{\sqrt{\mu_0}}\int_{\theta_0}^{p_0}\frac{t-\theta_0}{\sqrt{2(F(t)-F(\theta_0))}}\,\mathrm{d}t\right)\delta^{\infty},\notag\\
\drup\stackrel{R\to\infty}{-\!\!\!-\!\!\!-\!\!\!-\!\!\!-\!\!\!\rightharpoonup}&\left(\sqrt{\mu_0}\int_{\theta_0}^{p_0}\sqrt{2(F(t)-F(\theta_0))}\,\mathrm{d}t\right)\delta^{\infty}.\notag
\end{align}

\begin{theorem}[Boundary concentrations]\label{thm2}
Under the same hypotheses as in Theorem~\ref{thm1}, as $R\to\infty$, for any $r\in[0,\infty)$, there hold 
\begin{align}\label{0607-abc}
\dru(r)\to0\quad{and}\quad\drup(r)\to0\quad{as}\,\,{R\to\infty},
\end{align} 
and
\begin{align}
\lim_{R\to\infty}\int_0^{\infty}\dru(r)\,\mathrm{d}r=\,&\frac{1}{\sqrt{\mu_0}}\int_{\theta_0}^{p_0}\frac{t-\theta_0}{\sqrt{2(F(t)-F(\theta_0))}}\,\mathrm{d}t,\label{g1}\\
\lim_{R\to\infty}\int_0^{\infty}\drup(r)\,\mathrm{d}r=\,&\sqrt{\mu_0}\int_{\theta_0}^{p_0}\sqrt{2(F(t)-F(\theta_0))}\,\mathrm{d}t.\label{g2}
\end{align}
\end{theorem}
\begin{remark}\label{rk3}
We shall stress that \eqref{g1} is well-defined. Indeed, by (A1) it is easy to obtain
\begin{align*}
\frac{1}{p_0-\theta_0}\int_{\theta_0}^{p_0}\frac{t-\theta_0}{\sqrt{2(F(t)-F(\theta_0))}}\,\mathrm{d}t\in[(\ds\max_{[\theta_0,p_0]}f')^{-1/2},({\ds\min_{[\theta_0,p_0]}f'})^{-1/2}].
\end{align*}
\end{remark}

\subsection{A significant idea}
\noindent

To study the asymptotic behavior of $\bigu$ as $R\to\infty$, we consider a change of variables 
\begin{align}\label{intro-1}
\eps=\frac{1}{R}\to0+,\,\,s=\eps{r}\in(0,1],\,\,\ueps(s)=\bigu(r),\,\,\aeps(s)=\mta(r),\,\,\beps(s)=\mtb(r).
\end{align}
 In what follows, we use the symbol
 $$\doo:=\frac{\mbox{d}}{\mbox{d}s}$$
for the derivative with respect to the variable $s$ rather than $'$ to avoid the notation confusion with the prime notation $'$ for the derivative with respect to the variable $r$. Then we have 
\begin{align}\label{new-0208}
\us(s)=\eps^{-1}{\bigu'(r)}=R\bigu'(r),\,\,\alphas(s)=R\mta'(r),\,\,\betas(s)=R\mtb'(r),
\end{align}
 and (\ref{eq2})--(\ref{bd2}) is equivalent to the following singularly perturbed equation with small parameter $\eps$:
\begin{align}
\eps^2\left(\uss(s)+\left(\frac{N-1}{s}+\frac{\alphas(s)}{\aeps(s)}\right)\us(s)\right)=&\,\frac{\beps(s)}{\aeps(s)}f(\ueps(s)),\,\,s\in(0,1),\label{eq3}\\
\us(0)=0,\,\,\eps\us(1)=&\,\mte(\ueps(1)).\label{bd3}
\end{align}
Hence, the equation~(\ref{eq2}) in the domain $(0,R)$ with $R\to\infty$ becomes a singularly perturbed equation~(\ref{eq3}) with $\eps\downarrow0$ in a finite domain $(0,1)$. To deal with asymptotics of $\ueps$, one can multiply (\ref{eq3}) by $\doo{\ueps}$ and make simple calculations to obtain a first-order ODE
\begin{align}\label{1st-ode}
\frac{\eps^2}{2}&\left(\us(s))^2-\frac{\beps(s)}{\aeps(s)}F(\ueps(s))\right)\notag\\[-0.7em]
&\\[-0.7em]
&\,=-\int_{k^*}^s\left[\eps^2\left(\frac{N-1}{t}+\frac{\alphas(t)}{\aeps(t)}\right)(\us(t))^2+F(\ueps(t))\doo\left(\frac{\beps(t)}{\aeps(t)}\right)\right]\mbox{d}t+C_{k^*,\eps},\,\,s\in[k^*,1],\notag
\end{align}
with
\begin{align}\label{cstar}
 C_{k^*,\eps}=\frac{\eps^2}{2}(\us(k^*))^2-\frac{\beps(k^*)}{\aeps(k^*)}F(\ueps(k^*)),
\end{align}
where $\doo\left(\frac{\beps}{\aeps}\right):=\frac{\mbox{d}}{\mbox{d}t}\left(\frac{\beps}{\aeps}\right)$ and $F$ is defined in (\ref{big-f}). In particular, (\ref{1st-ode}) together with the boundary condition~(\ref{bd3}) implies
\begin{align}\label{1st-ode-0210-neww}
-\frac{1}{2}\big(\mte(\ueps(1))\big)^2&+\frac{\beps(1)}{\aeps(1)}F(\ueps(1))\notag\\[-0.7em]
&\\[-0.7em]
=\int_{k^*}^1&\left[\eps^2\left(\frac{N-1}{t}+\frac{\alphas(t)}{\aeps(t)}\right)(\us(t))^2+F(\ueps(t))\doo\left(\frac{\beps(t)}{\aeps(t)}\right)\right]\mbox{d}t-C_{k^*,\eps}.\notag
\end{align}
We will show that the right-hand side of (\ref{1st-ode-0210-neww}) tends to zero as $\eps\downarrow0$. Its precise leading term plays a key role in the asymptotics of $\ueps(1)$.

The remainder of the paper proceeds as follows. In the next section, we will establish the interior and gradient estimate of $\ueps$ in Lemmas~\ref{lem1} and \ref{lem2}, which give the precise leading order term of the expression in the right-hand side of (\ref{1st-ode-0210-neww}). In particular, by (\ref{mta-b-0211}), (\ref{intro-1}) and (\ref{1st-ode-0210-neww}), we obtain 
\begin{align}\label{0403-add}
(\mte(\ueps(1)))^2=2\mu_0\left(F(\ueps(1))-F(\theta_0)\right)+o_{\eps}(1)\,\,\mathrm{as}\,\,\eps\downarrow0.
\end{align}
As will be mentioned later on, the interior estimate (\ref{0209-est1}) and the gradient estimate (\ref{0209-est2}) show that if $\ds\lim_{\eps\downarrow0}\frac{1-s_{\eps}}{\eps}=\infty$, there still hold $\ueps(s_{\eps})\to\theta_0$ and $\us(s_{\eps})\to0$ exponentially as $\eps$ goes to zero. Furthermore, in Theorem~\ref{lem3}, we combine (\ref{1st-ode-0210-neww}) with (\ref{pohozaev-id})--(\ref{0214-nthu-1}) to establish the {\bf precise leading order terms} of (\ref{0403-add}) as follows (see (\ref{0219-0754-newadd}) also):   
\begin{align}\label{0219-0754}
\frac{1}{\eps}&\left(-\frac{1}{2}(\mte(\ueps(1)))^2+\mu_0\left(F(\ueps(1))-F(\theta_0)\right)\right)\notag\\
&\quad=\sqrt{\mu_0}\left((N-1)+\frac{\alphas(1)}{2\aeps(1)}+\frac{\betas(1)}{2\beps(1)}\right)\int_{\theta_0}^{p_0}\sqrt{2(F(t)-F(\theta_0))}\,\mathrm{d}t+o_{\eps}(1),\notag
\end{align} 
which will determine the precise first two order terms of $\ueps(1)$ and $\us(1)$ with respect to small $\eps>0$. We shall highlight here that Theorem~\ref{lem3} plays a key role in the proof of the main theorems. The proof of Theorems~\ref{thm1} and \ref{thm2} and Corollary~\ref{rk1} will be stated in Section~\ref{sec-thm1}. To see the effect of the perturbation of $\frac{\mtb(R)}{\mta(R)}$ around $\mu_0$ on solution asymptotics, in the final Section~\ref{sec-ap} we replace the strong assumption \eqref{mta-b-0211} with $\ds\liminf_{R\to\infty}$~$R(\frac{\mtb(R)}{\mta(R)}-\mu_0)>0$ which includes the situation~\eqref{0608-hap}. Then, we establish in Corollary~\ref{cor0603} the precise effect of $\frac{\mtb(R)}{\mta(R)}-\mu_0$ on asymptotics of $\bigu(R)$ and $\bigu'(R)$.

\section{Proof of the main results}\label{sec-pre}
\noindent

In this section, we first investigate asymptotics for solutions $\ueps$ of the equation~(\ref{eq3})--(\ref{bd3}) and establish the corresponding boundary gradient asymptotic expansions as $\eps$ tends to zero. Such asymptotics play a crucial role in the asymptotic expansions of $\bigu$ and $\bigu'$ as $R$ approaches infinity. In Section~\ref{sec-thm1} we shall complete the proof of Theorems~\ref{thm1} and \ref{thm2} and Corollary~\ref{rk1}.

\subsection{Interior estimates}
\noindent

To go further, let us state some properties which can be obtained directly from (A1)--(A3), (\ref{id-mtu}), (\ref{ali88}) and (\ref{intro-1})--(\ref{bd3}).
\begin{itemize}
\item[\textbf{(P1).}] As $\eps>0$ is sufficiently small, we have
\begin{align*}
\frac{\aeps(s)}{\beps(s)}\geq\frac{1}{2}\lim_{R\to\infty}\inf_{[0,R]}\frac{\mta(r)}{\mtb(r)}\quad\mathrm{and}\quad\frac{\beps(s)}{\aeps(s)}\geq\frac{1}{2}\lim_{R\to\infty}\inf_{[0,R]}\frac{\mtb(r)}{\mta(r)},\,\,\forall\,s\in[0,1].
\end{align*}
Henceforth we set $\ds{C}_1:=\frac{1}{2}\min\left\{\lim_{R\to\infty}\inf_{[0,R]}\frac{\mta(r)}{\mtb(r)},\lim_{R\to\infty}\inf_{[0,R]}\frac{\mtb(r)}{\mta(r)}\right\}>0$. Along with (A3) gives  
\begin{equation*}
\min_{s\in[0,1]}\frac{\aeps(s)}{\beps(s)}\geq{C}_1,\,\,\min_{s\in[0,1]}\frac{\beps(s)}{\aeps(s)}\geq{C}_1\,\,\mbox{as}\,\,0<\eps\ll1.
\end{equation*}
\item[\textbf{(P2).}] As $\eps>0$ is sufficiently small, 
\begin{align*}
\sup_{s\in[k^*,1]}\left(\frac{\left|(\doo\aeps)(s)\right|}{\aeps(s)}+\frac{\left|(\doo\beps)(s)\right|}{\beps(s)}+\frac{\left|(\doo^2\aeps)(s)\right|}{\alpha_{\eps}^2(s)}\right)\leq{C}_2,  
\end{align*}
where $k^*\in(0,1)$ is defined in (A3) and $C_2$ is a positive constant independent of $\eps$.
\item[\textbf{(P3).}] $\ueps-\theta_0$ and $\doo\ueps$ are non-negative in $[0,1]$. Moreover, by (A1) we have
\begin{align*}
 f'(\ueps(s))\geq{C}_3\,\,\mathrm{and}\,\,f(\ueps(s))(\ueps(s)-\theta_0)\geq{C}_3(\ueps(s)-\theta_0)^2,\quad\forall\,s\in[0,1], 
\end{align*}
where $C_3$ is a positive constant independent of $\eps$.
\item[\textbf{(P4).}] By (\ref{bd1}) and (A2), we have
\begin{align*}
{\eps}^{-1}{\mte(\max_{[0,1]}\ueps)}\leq\us(1)\leq{\eps}^{-1}{\mte(\theta_0)}.
\end{align*}
\item[\textbf{(P5).}] By (\ref{eq3}) and $\ueps\geq\theta_0$, we have
\begin{align*}
 \doo\left(s^{N-1}\aeps(s)\us(s)\right)=s^{N-1}\beps(s)f(\ueps(s))\geq0,\,\,\forall\,s\in(0,1). 
\end{align*}
Hence, $s^{N-1}\aeps(s)\us(s)$ is increasing to $s\in[0,1]$.
\end{itemize}
  Moreover, we have the following estimates of $\ueps$ and $\doo{\ueps}$ with respect to sufficiently small $\eps>0$.

\begin{lemma}\label{lem1}
Assume that (A1)--(A3) hold. For $\eps>0$ and $\aeps$ and $\beps$ satisfying (\ref{intro-1}), let $\ueps\in\mathrm{C}^1((0,1])\cap\mathrm{C}^{\infty}((0,1))$ be the unique solution of (\ref{eq3})--(\ref{bd3}). Then there exist positive constants $\eps^*$ and $M^*$ independent of $\eps$ such that as $0<\eps<\eps^*$, 
\begin{align}\label{0209-est1}
0\leq\ueps(s)-\theta_0\leq2(\ueps(1)-\theta_0)e^{-\frac{M^*}{\eps}(1-s)},
\end{align}
and
\begin{align}\label{0209-est2}
0\leq{s}^{N-1}\aeps(s)\us(s)\leq\frac{2}{\eps}\aeps(1)\mte(\theta_0)e^{-\frac{M^*}{\eps}(1-s)},
\end{align}
for $s\in[0,1]$.
\end{lemma}
\begin{proof}
We first deal with the estimate of $\ueps(s)-\theta_0$. Multiplying (\ref{eq3}) by $\ueps(s)-\theta_0$ and using (P1) and (P3), we obtain
\begin{align}\label{0209-939pm}
\eps^2\left(\uss(s)+\left(\frac{N-1}{s}+\frac{\alphas(s)}{\aeps(s)}\right)\us(s)\right)(\ueps(s)-\theta_0)
\geq\,C_1C_3(\ueps(s)-\theta_0)^2.
\end{align}
One can further check that, for $s\in[k^*,1]$,
\begin{align}\label{cc-0209-2019}
&\left(\uss(s)+\left(\frac{N-1}{s}+\frac{\alphas(s)}{\aeps(s)}\right)\us(s)\right)(\ueps(s)-\theta_0)\notag\\
=&\frac{1}{2}\doo^2((\ueps(s)-\theta_0)^2)-\left(\doo(\ueps(s)-\theta_0)\right)^2+\left(\frac{N-1}{s}+\frac{\alphas(s)}{\aeps(s)}\right)\left(\doo(\ueps(s)-\theta_0)\right)(\ueps(s)-\theta_0)\\
\leq&\frac{1}{2}\doo^2((\ueps(s)-\theta_0)^2)+\frac{1}{4}\left(\frac{N-1}{k^*}+C_2\right)^2(\ueps(s)-\theta_0)^2.\notag
\end{align}
Here we have used (P2), (P3) and $\ueps(s)\geq\theta_0$ to deal with the last inequality of (\ref{cc-0209-2019}). Combining (\ref{0209-939pm}) with (\ref{cc-0209-2019}), one finds
\begin{align}\label{0209-1021pm}
\eps^2\doo^2((\ueps(s)-\theta_0)^2)\geq&\left[2C_1C_3-\frac{\eps^2}{2}\left(\frac{N-1}{k^*}+C_2\right)^2\right](\ueps(s)-\theta_0)^2\notag\\
\geq&\,C_1C_3(\ueps(s)-\theta_0)^2,\,\,s\in[k^*,1],
\end{align}
as
\begin{align*}
 0<\eps\leq{\sqrt{2C_1C_3}}\left(\frac{N-1}{k^*}+C_2\right)^{-1}.
\end{align*}
 Consequently, applying the standard PDE comparison theorem to (\ref{0209-1021pm}), we may arrive at the estimate 
\begin{align}\label{1041pm-0209}
0\leq\ueps(s)-\theta_0\leq(\ueps(1)-\theta_0)\left(e^{-\frac{\sqrt{C_1C_3}}{2\eps}(s-k^*)}+e^{-\frac{\sqrt{C_1C_3}}{2\eps}(1-s)}\right),\,\,\forall\,s\in[k^*,1].
\end{align}
Now we shall refine the estimate \eqref{1041pm-0209}. Firstly, we assume $s\in[\frac{k^*+1}{2},1]$, i.e., $s-k^*\geq1-s$. Then (\ref{1041pm-0209}) implies 
\begin{align}\label{alb-1-0209}
0\leq\ueps(s)-\theta_0\leq2(\ueps(1)-\theta_0)e^{-\frac{\sqrt{C_1C_3}}{2\eps}(1-s)}.
\end{align}
On the other hand, for $s\in[0,\frac{k^*+1}{2}]$, by the property $\us(s)\geq0$ and (\ref{1041pm-0209}) we have
\begin{align}\label{alb-2-0209}
0\leq\ueps(s)-\theta_0\leq&\ueps(\frac{k^*+1}{2})-\theta_0\leq\,2(\ueps(1)-\theta_0)e^{-\frac{\sqrt{C_1C_3}}{4\eps}(1-k^*)}\notag\\[-0.7em]
&\\[-0.7em]
\leq&\,2(\ueps(1)-\theta_0)e^{-\frac{(1-k^*)\sqrt{C_1C_3}}{4\eps}(1-s)}.\notag
\end{align}
It therefore follows from (\ref{alb-1-0209}) and (\ref{alb-2-0209}) that
\begin{align}\label{alb-3-0209}
0\leq\ueps(s)-\theta_0\leq2(\ueps(1)-\theta_0)e^{-\frac{(1-k^*)\sqrt{C_1C_3}}{4\eps}(1-s)},\,\,\forall\,s\in[0,1].
\end{align}

Now we shall deal with the estimate of $\doo\ueps$. Multiplying (\ref{eq3}) by $s^{N-1}\aeps(s)$ and taking the derivative of the expression with respect to the variable $s$, one arrives at 
\begin{align}\label{0609-my}
\eps^2\doo^2\left(s^{N-1}\aeps(s)\us(s)\right)=\doo\left(s^{N-1}\beps(s)\right)f(\ueps(s))+s^{N-1}\beps(s)f'(\ueps(s))\us(s).
\end{align}
To deal with the left-hand side of \eqref{0609-my}, we first notice $\doo\left(s^{N-1}\beps(s)\right)\geq0$ (by (A3)). Thanks to  (P1) and (P3), we arrive at a differential inequality
\begin{align}\label{aha0210}
\eps^2\doo^2\left(s^{N-1}\aeps(s)\us(s)\right)\geq&\,{C_3}\left(\inf_{s\in[0,1]}\frac{\beps(s)}{\aeps(s)}\right)\left(s^{N-1}\aeps(s)\us(s)\right)\notag\\[-0.7em]
&\\[-0.7em]
\geq&\,{C_1C_3}s^{N-1}\aeps(s)\us(s),\,\,\mbox{in}\,\,(0,1).\notag
\end{align}
Applying the standard PDE comparison theorem to (\ref{aha0210}) and using (P4) immediately gives
\begin{align}\label{aha-aha-0210}
0\leq{s}^{N-1}\aeps(s)\us(s)\leq\frac{\aeps(1)\mte(\theta_0)}{\eps}\left(e^{-\frac{\sqrt{C_1C_3}}{\eps}s}+e^{-\frac{\sqrt{C_1C_3}}{\eps}(1-s)}\right).
\end{align}
Along with the fact that ${s}^{N-1}\aeps(s)\us(s)$ is increasing to $s$ (see (P5)), we may follow the similar argument as in (\ref{1041pm-0209})--(\ref{alb-3-0209}) to obtain
\begin{align}\label{aha-aha-0223}
0\leq{s}^{N-1}\aeps(s)\us(s)\leq\frac{2}{\eps}\aeps(1)\mte(\theta_0)e^{-\frac{(1-k^*)\sqrt{C_1C_3}}{2\eps}(1-s)}.
\end{align}
Let us set $M^*=\frac{(1-k^*)\sqrt{C_1C_3}}{4}$. Then, (\ref{0209-est1}) and (\ref{0209-est2}) follow from (\ref{alb-3-0209}) and (\ref{aha-aha-0223}), respectively. This completes the proof of Lemma~\ref{lem1}.
\end{proof}

The following result states the uniform boundedness of $\ueps$ and the leading order terms of $\ueps(1)$ and $\us(1)$ with respect to $0<\eps\ll1$.
\begin{lemma}\label{lem2}
Under the same hypotheses as in Lemma~\ref{lem1}, $\ds\max_{[0,1]}\ueps=\ueps(1)$ is uniformly bounded as $\eps>0$ is sufficiently small. In particular, as $\eps\downarrow0$, for each $s\in[0,1)$ independent of $\eps$, $|\ueps(s)-\theta_0|+\eps|\us(s)|\to0$ exponentially, and
\begin{align}\label{0210-ueps-1}
\ueps(1)\to{p}\,\,and\,\,\eps\us(1)\to\mte(p_0),
\end{align}
where $p$ is the unique root of (\ref{equ-p}). Moreover, 
\begin{align}\label{0212-eq}
\left|\eps\us(s)-\sqrt{\frac{2\beps(s)}{\aeps(s)}\left(F(\ueps(s))-F(\theta_0)\right)}\right|\leq\widetilde{C}\eps^{1/2},\,\,for\,\,s\in[k^*,1],
\end{align}
where $\widetilde{C}$ is a positive constant independent of $\eps$.
\end{lemma}
\begin{proof}
We first claim $\ds\limsup_{\eps\downarrow0}\ueps(1)<\infty$. Integrating (\ref{0209-est2}) over the interval $(k^*,1)$, one obtains
\begin{equation*}
\ds(k^*)^{N-1}\left(\min_{[k^*,1]}\aeps\right)(\ueps(1)-\ueps(k^*))\leq\frac{\aeps(1)\mte(\theta_0)}{M^*}.
\end{equation*}
Along with (\ref{0209-est1}), one arrives at
\begin{align*}
\ueps(1)-\frac{\aeps(1)\mte(\theta_0)}{\ds{M}^*(k^*)^{N-1}\min_{[k^*,1]}\aeps}\leq\ueps(k^*)\leq\theta_0+2(\ueps(1)-\theta_0)e^{-\frac{M^*}{\eps}(1-k^*)}.
\end{align*}
Because $M^*>0$, $k^*<1$ and $\frac{\aeps(1)}{\ds\min_{[k^*,1]}\aeps}$ is uniformly bounded to $0<\eps\ll1$ (by (A3) and (P1)), the above inequality implies
\begin{align}\label{uone-up}
\limsup_{\eps\downarrow0}\ueps(1)\leq\theta_0+\frac{\aeps(1)\mte(\theta_0)}{\ds{M}^*(k^*)^{N-1}\min_{[k^*,1]}\aeps}<\infty.
\end{align}
Since $\ueps(1)$ is uniformly bounded as $0<\eps\ll1$, and $\theta_0\leq\ueps(s)\leq\ueps(1)$, we immediately obtain the uniform boundedness of $\ueps$ as $0<\eps\ll1$. Moreover, (\ref{0209-est1}) can be improved by
\begin{align}\label{0211-est3}
0\leq\ueps(s)-\theta_0\leq{L}_{\eps}e^{-\frac{M^*}{\eps}(1-s)},
\end{align}
as $0<\eps\ll1$, where 
\begin{align}\label{0607-after}
{L}_{\eps}:=1+\theta_0+\frac{\aeps(1)\mte(\theta_0)}{\ds{M}^*(k^*)^{N-1}\min_{[k^*,1]}\aeps}.
\end{align}
 Note that $L_{\eps}$ is uniformly bounded to $\eps>0$. Consequently, by (\ref{0209-est1}) and (\ref{0211-est3}), we show that for each $s\in[0,1)$ independent of $\eps$, both $|\ueps(s)-\theta_0|$ and $\eps|\us(s)|$ decay to zero exponentially as $\eps$ approaches zero.

To prove (\ref{0210-ueps-1}), we shall obtain the precise leading order terms of $\ueps(1)$ and $\us(1)$ with respect to small $\eps$. Let us first deal with the terms in the right-hand side of (\ref{1st-ode-0210-neww}). Firstly, by (P2) and (\ref{0209-est2}) one may check that, as $0<\eps\ll1$,
\begin{align}\label{hei-2019-916pm}
\int_{k^*}^1\eps^2&\left(\frac{N-1}{t}+\frac{\alphas(t)}{\aeps(t)}\right)(\us(t))^2\mbox{d}t\notag\\
\leq&\int_{k^*}^1\left(\frac{N-1}{t}+\frac{\alphas(t)}{\aeps(t)}\right)\left(\frac{\aeps(1)\mte(\theta_0)}{{t}^{N-1}\aeps(t)}\right)^2\left(e^{-\frac{M^*}{\eps}t}+e^{-\frac{M^*}{\eps}(1-t)}\right)^2\mbox{d}t\\
\leq&\,\frac{2}{M^*}\left(\frac{\aeps(1)\mte(\theta_0)}{\ds{(k^*)}^{N-1}\min_{[k^*,1]}\aeps}\right)^2\left(\frac{N-1}{k^*}+C_2\right)\eps:=C_4\eps.\notag
\end{align}
Note that $C_4$ is a positive constant independent of $\eps$ due to (A3) and (\ref{new-0208}). Next, we shall claim 
\begin{align*}
\int_{k^*}^1F(\ueps(t))\doo\left(\frac{\beps(t)}{\aeps(t)}\right)\mbox{d}t-C_{k^*,\eps}\sim\,F(\theta_0)\frac{\beps(1)}{\aeps(1)},\,\,\mathrm{as}\,\,0<\eps\ll1.
\end{align*}
 By using (\ref{big-f}), (\ref{cstar}) (P1)--(P3), (\ref{0209-est2}) and (\ref{0211-est3}), we have
\begin{align}\label{0211-morning-1}
&\left|\int_{k^*}^1F(\ueps(t))\doo\left(\frac{\beps(t)}{\aeps(t)}\right)\mbox{d}t-C_{k^*,\eps}-F(\theta_0)\frac{\beps(1)}{\aeps(1)}\right|\notag\\
&\quad\quad\quad\leq\left|C_{k^*,\eps}+\frac{\beps(k^*)}{\aeps(k^*)}F(\theta_0)\right|+\left|\int_{k^*}^1F(\ueps(t))\doo\left(\frac{\beps(t)}{\aeps(t)}\right)\mbox{d}t-F(\theta_0)\left(\frac{\beps(1)}{\aeps(1)}-\frac{\beps(k^*)}{\aeps(k^*)}\right)\right|\notag\\
&\quad\quad\quad\leq\left|C_{k^*,\eps}+\frac{\beps(k^*)}{\aeps(k^*)}F(\theta_0)\right|+\int_{k^*}^1\left|F(\ueps(t))-F(\theta_0)\right|\left|\doo\left(\frac{\beps(t)}{\aeps(t)}\right)\right|\mbox{d}t\\
&\quad\quad\quad\leq\,{C}_5\left(e^{-\frac{M^*}{\eps}k^*}+e^{-\frac{M^*}{\eps}(1-k^*)}\right)+{C}_2\left(1+\max_{t\in[0,1]}\frac{\aeps(t)}{\beps(t)}\right)f(\ueps(1))\int_{k^*}^1(\ueps(t)-\theta_0)\mbox{d}t\notag\\
&\quad\quad\quad\leq\,{C}_5\left(e^{-\frac{M^*}{\eps}k^*}+e^{-\frac{M^*}{\eps}(1-k^*)}\right)+C_6\eps,\notag
\end{align}
as $0<\eps\ll1$, where $C_5$ is a positive constant independent of $\eps$, and $C_6$ can be any large positive constant satisfying
\begin{align*}
C_6>\frac{2C_2}{M^*}\limsup_{\eps\downarrow0}\left\{L_{\eps}\left(1+\max_{t\in[0,1]}\frac{\aeps(t)}{\beps(t)}\right)f(L_{\eps}+\theta_0)\right\}. 
\end{align*}
Here we have used (\ref{cstar}), (\ref{0209-est2}) and (\ref{0211-est3}) to get
\begin{align}\label{0211-morning-2}
\left|C_{k^*,\eps}+\frac{\beps(k^*)}{\aeps(k^*)}F(\theta_0)\right|\leq{C}_5\left(e^{-\frac{M^*}{\eps}k^*}+e^{-\frac{M^*}{\eps}(1-k^*)}\right)
\end{align}
which verifies the last second line of \eqref{0211-morning-1}.
Combining (\ref{1st-ode-0210-neww}) with (\ref{0211-morning-1}) yields
\begin{align}\label{0211-mor-lim}
{(\mte(\ueps(1)))^2}-\frac{{2}\beps(1)}{\aeps(1)}(F(\ueps(1))-F(\theta_0))\stackrel{\eps\downarrow0}{\longrightarrow}0.
\end{align} 
On the other hand, by (\ref{mta-b-0211}) and (\ref{intro-1}), we have $\frac{\beps(1)}{\aeps(1)}\to\mu_0$ as $\eps\downarrow0$. Note that $F$ is strictly increasing in $(\theta_0,\infty)$. Since $\ueps(1)\geq\theta_0$ and $\mte>0$ is a decreasing function (cf. (A2)), we obtain $\lim_{\eps\downarrow0}\ueps(1)=p$ which uniquely solves (\ref{equ-p}). Moreover, by this with the boundary condition~(\ref{bd3}), we have $\lim_{\eps\downarrow0}\eps\us(1)=\mte(p_0)$. Therefore, we obtain (\ref{0210-ueps-1}).

It remains to prove (\ref{0212-eq}). Let $s\in[k^*,1]$. Following the similar arguments as in (\ref{hei-2019-916pm}) and (\ref{0211-morning-1}), we can get estimates 
\begin{align}\label{0212-10pm-1}
\int_{k^*}^s\eps^2&\left(\frac{N-1}{t}+\frac{\alphas(t)}{\aeps(t)}\right)(\us(t))^2\mbox{d}t
\leq\,C_7\left(\frac{N-1}{k^*}+C_2\right)\eps
\end{align}
and
\begin{align}\label{0212-10pm-2}
\left|\int_{k^*}^sF(\ueps(t))\doo\left(\frac{\beps(t)}{\aeps(t)}\right)\mbox{d}t-C_{k^*,\eps}-F(\theta_0)\frac{\beps(s)}{\aeps(s)}\right|
\leq\,C_8\eps,
\end{align}
as $0<\eps\ll1$, where $C_7,\,C_8>0$ independent of $s$ and $\eps$. Then by (\ref{1st-ode}) and (\ref{0212-10pm-1})--(\ref{0212-10pm-2}), we arrive at 
\begin{align}\label{lem2end}
\left|\eps^2\left(\us(s)\right)^2-\frac{2\beps(s)}{\aeps(s)}\left(F(\ueps(s))-F(\theta_0)\right)\right|\leq{\widetilde{C}}^2\eps
\end{align}
with a positive constant $\widetilde{C}$ independent of $s$ and $\eps$. Since $\us(s)\geq0$ and
 $F(\ueps(s))\geq{F}(\theta_0)$, $\forall{s}\in[0,1]$ (see (P3)), together with (\ref{lem2end}) we immediately get (\ref{0212-eq}) and complete the proof Lemma~\ref{lem2}.
\end{proof}

\subsection{Boundary asymptotics with precise first two order terms}
\noindent

Recall that (\ref{hei-2019-916pm}) and (\ref{0211-morning-1}) imply
\begin{align*}
\sup_{0<\eps\ll1}\eps\int_{k^*}^1\left|\frac{N-1}{s}+\frac{\alphas(s)}{\aeps(s)}\right|\left(\us(s)\right)^2\mathrm{d}s&<\infty,\\
\sup_{0<\eps\ll1}\frac{1}{\eps}\left|\int_{k^*}^1F(\ueps(t)) \left(\doo(\frac{\beps}{\aeps})\right)\!\!(t)\,\mbox{d}t-C_{k^*,\eps}-F(\theta_0)\frac{\beps(1)}{\aeps(1)}\right|&<\infty.
\end{align*} 
 To obtain the  structure of the solution $\ueps$, we further establish their precise leading order terms which play a crucial role in the refined asymptotics of $\ueps(1)$ and $\left(\doo\ueps\right)(1)$. The asymptotics are stated as follows.
\begin{theorem}\label{lem3}
Under the same hypotheses as in Lemma~\ref{lem1}, for $\eps>0$ sufficiently small, 
\begin{align}\label{pohozaev-id}
\eps\int_{k^*}^1&\left(\frac{N-1}{s}+\frac{\alphas(s)}{\aeps(s)}\right)(\us(s))^2\mathrm{d}s\notag\\[-0.7em]
&\\[-0.7em]
=&\sqrt{\mu_0}\left((N-1)+\frac{\alphas(1)}{\aeps(1)}\right)\int_{\theta_0}^{p_0}\sqrt{2(F(t)-F(\theta_0))}\,\mathrm{d}t+o_{\eps}(1),\notag
\end{align}
and
\begin{align}\label{0214-nthu-1}
\frac{1}{\eps}\left(\int_{k^*}^1\right.&\left.F(\ueps(s)) \doo\left(\frac{\beps(s)}{\aeps(s)}\right)\mathrm{d}s-C_{k^*,\eps}-F(\theta_0)\frac{\beps(1)}{\aeps(1)}\right)\notag\\[-0.7em]
&\\[-0.7em]
&=\frac{1}{2\aeps(1)}\left(\frac{\betas(1)}{\sqrt{\mu_0}}-\sqrt{\mu_0}\alphas(1)\right)\int_{\theta_0}^{p_0}\sqrt{2(F(t)-F(\theta_0))}\,\mathrm{d}t+o_{\eps}(1),\notag
\end{align}
where $o_{\eps}(1)$ denotes the quantity approaching zero as $\eps\downarrow0$. 
\end{theorem}
\begin{proof}
Let us fix a number $\tau_a\in(0,1)$ independent of $\eps$. By (P1) and (P2), we obtain
\begin{align}
\sup_{s\in[1-\eps^{\tau_a},1]}\left|\frac{\alphas(s)}{\aeps(s)}-\frac{\alphas(1)}{\aeps(1)}\right|\leq\eps^{\tau_a}\sup_{[1-\eps^{\tau_a},1]}\left|\doo\left(\frac{\doo\aeps}{\aeps}\right)\right|
\leq\left(C_2^2+C_2\sup_{[1-\eps^{\tau_a},1]}\aeps\right)\eps^{\tau_a}\stackrel{\eps\downarrow0}{\longrightarrow}0,&\label{a-a}\\[0.5em]
\sup_{s\in[1-\eps^{\tau_a},1]}\left|\frac{\beps(s)}{\aeps(s)}-\frac{\beps(1)}{\aeps(1)}\right|\leq\eps^{\tau_a}\sup_{[1-\eps^{\tau_a},1]}\left|\doo\left(\frac{\beps}{\aeps}\right)\right|
\leq\left(1+\frac{1}{C_1}\right)C_2\eps^{\tau_a}\stackrel{\eps\downarrow0}{\longrightarrow}0.&\label{b-a}
\end{align}
Hence, for $0<\eps\ll1$, we consider a decomposition
\begin{align}\label{0214-ily}
\int_{k^*}^1\eps^2&\left(\frac{N-1}{s}+\frac{\alphas(s)}{\aeps(s)}\right)(\us(s))^2\mbox{d}s\notag\\
=&\int_{k^*}^{1-{\eps}^{\tau_a}}\eps^2\left(\frac{N-1}{s}+\frac{\alphas(s)}{\aeps(s)}\right)(\us(s))^2\mbox{d}t\notag\\[-0.7em]
&\\[-0.7em]
&\,+\int_{1-{\eps}^{\tau_a}}^1\eps^2\left[\left(\frac{N-1}{s}+\frac{\alphas(s)}{\aeps(s)}\right)-\left((N-1)+\frac{\alphas(1)}{\aeps(1)}\right)\right](\us(s))^2\mbox{d}s\notag\\
&\,+\eps^2\left((N-1)+\frac{\alphas(1)}{\aeps(1)}\right)\int_{1-{\eps}^{\tau_a}}^1(\us(s))^2\mbox{d}s.\notag
\end{align}
Using the gradient estimate (\ref{0209-est2}) and (\ref{a-a}), we may follow the similar argument as in (\ref{hei-2019-916pm}) to get
\begin{align}
\left|\int_{k^*}^{1-{\eps}^{\tau_a}}\eps^2\left(\frac{N-1}{s}+\frac{\alphas(s)}{\aeps(s)}\right)(\us(s))^2\,\mbox{d}t\right|&\notag\\
\leq2\left(\frac{N-1}{k^*}+C_2\right)\left(\frac{\aeps(1)\mte(\theta_0)}{\ds{(k^*)}^{N-1}\min_{[k^*,1]}\aeps}\right)^2&\int_{k^*}^{1-{\eps}^{\tau_a}}\left(e^{-\frac{2M^*}{\eps}s}+e^{-\frac{2M^*}{\eps}(1-s)}\right)\mbox{d}s\notag\\
(\mbox{due\,\,to}\,\,\tau_a\in(0,1))\,\leq{C}_9\left(\frac{N-1}{k^*}+C_2\right){\eps}e^{-{2M^*}{\eps^{\tau_a-1}}}\ll\eps\,\,\mbox{as}\,&\,\,\,0<\eps\ll1,\notag
\end{align}
and
\begin{align*}
\left|\int_{1-{\eps}^{\tau_a}}^1\eps^2\left[\left(\frac{N-1}{s}+\frac{\alphas(s)}{\aeps(s)}\right)-\left((N-1)+\frac{\alphas(1)}{\aeps(1)}\right)\right](\us(s))^2\,\mbox{d}s\right|&\\
\leq\,C_{10}\eps^{\tau_a}\int_{1-{\eps}^{\tau_a}}^1\eps^2(\us(s))^2&\,\mbox{d}s\ll\eps\,\,\mbox{as}\,\,0<\eps\ll1,
\end{align*}
where $C_9$ and $C_{10}$ are positive constants independent of $\eps$. 

To deal with the last term of (\ref{0214-ily}), let us rewrite (\ref{0212-eq}) as
\begin{align}\label{0213-eq-ni}
\eps\us(s)=\sqrt{\frac{2\beps(s)}{\aeps(s)}\left(F(\ueps(s))-F(\theta_0)\right)}+\geps(s)\,\,\,\mathrm{and}\,\,\,|\geps(s)|\leq\widetilde{C}\eps^{1/2},\,\,\,\forall\,s\in[k^*,1].
\end{align}
Then by (\ref{b-a}) and (\ref{0213-eq-ni}) one may check that
\begin{align}\label{0214-nigi3}
\eps\int_{1-{\eps}^{\tau_a}}^1(\us(s))^2\,\mbox{d}s
=&\int_{1-{\eps}^{\tau_a}}^1\left(\sqrt{\frac{2\beps(s)}{\aeps(s)}\left(F(\ueps(s))-F(\theta_0)\right)}+\geps(s)\right)\us(s)\mbox{d}s\notag\\
=&\sqrt{\frac{\beps(1)}{\aeps(1)}}\int_{\ueps(1-{\eps}^{\tau_a})}^{\ueps(1)}\sqrt{2(F(t)-F(\theta_0))}\,\mbox{d}t+o_{\eps}(1)\notag\\[-0.7em]
&\\[-0.7em]
=&\sqrt{\frac{\beps(1)}{\aeps(1)}}\left\{\int_{\ueps(1-{\eps}^{\tau_a})}^{\theta_0}+\int_{\theta_0}^{p_0}+\int_{p}^{\ueps(1)}\right\}\sqrt{2(F(t)-F(\theta_0))}\,\mbox{d}t+o_{\eps}(1)\notag\\
=&\sqrt{\frac{\beps(1)}{\aeps(1)}}\int_{\theta_0}^{p_0}\sqrt{2(F(t)-F(\theta_0))}\,\mbox{d}t+o_{\eps}(1).\notag
\end{align}
Here we have used the following three estimates to deal with \eqref{0214-nigi3}:
\begin{align*}
\left|\int_{1-{\eps}^{\tau_a}}^1\geps(s)\us(s)\mbox{d}s\right|
\leq\widetilde{C}\eps^{1/2}\int_{1-{\eps}^{\tau_a}}^1\us(s)\mbox{d}s\leq\widetilde{C}\eps^{1/2}(\ueps(1)-\theta_0)\lesssim\eps^{1/2},
\end{align*}
\begin{align*}
\left|\int_{1-{\eps}^{\tau_a}}^1\sqrt{2\left(\frac{\beps(s)}{\aeps(s)}-\frac{\beps(1)}{\aeps(1)}\right)\left(F(\ueps(s))-F(\theta_0)\right)}\us(s)\mbox{d}s\right|&\\
(\mbox{by\,\,(\ref{b-a})})\,\,\leq\sqrt{2\left(1+\frac{1}{C_1}\right)C_2}\eps^{{\tau_a}/{2}}\left(F(\ueps(1))-F(\theta_0)\right)(\ueps(1)&\,-\theta_0)\lesssim\eps^{{\tau_a}/{2}},
\end{align*}
and 
\begin{align*}
\left|\left\{\int_{\ueps(1-{\eps}^{\tau_a})}^{\theta_0}+\int_{p}^{\ueps(1)}\right\}\sqrt{F(t)-F(\theta_0)}\,\mbox{d}t\right|&\\
(\mbox{by\,\,(\ref{0210-ueps-1})\,\,and\,\,(\ref{0211-est3})})\,\,\leq\sqrt{F(\ueps(1))-F(\theta_0)}\big(|\ueps(1-&\,{\eps}^{\tau_a})-\theta_0|+|\ueps(1)-p|\big)\stackrel{\eps\downarrow0}{\longrightarrow}0.
\end{align*}
 Since $\frac{\beps(1)}{\aeps(1)}\to\mu_0$ as $\eps\downarrow0$, (\ref{0214-nigi3}) immediately implies (\ref{pohozaev-id}).

Now we shall prove (\ref{0214-nthu-1}). From the first three lines of (\ref{0211-morning-1}), we obtain
\begin{align}\label{0614-hp}
\frac{1}{\eps}\left|\left(\int_{k^*}^1F(\ueps(s))\doo\right.\right.&\left.\left(\frac{\beps(s)}{\aeps(s)}\right)\mathrm{d}s-C_{k^*,\eps}-F(\theta_0)\frac{\beps(1)}{\aeps(1)}\right)\notag\\[-0.7em]
&\\[-0.7em]
&\left.-\int_{k^*}^1\left(F(\ueps(s))-F(\theta_0)\right)\doo\left(\frac{\beps(s)}{\aeps(s)}\right)\mbox{d}s\right|\lesssim\frac{1}{\eps}\left(e^{-\frac{M^*}{\eps}k^*}+e^{-\frac{M^*}{\eps}(1-k^*)}\right)\stackrel{\eps\downarrow0}{\longrightarrow}0,\notag
\end{align}
 Hence, by (\ref{0213-eq-ni}) and \eqref{0614-hp}, one finds
\begin{align}\label{0218-1050}
&\frac{1}{\eps}\left(\int_{k^*}^1F(\ueps(s)) \doo\left(\frac{\beps(s)}{\aeps(s)}\right)\,\mathrm{d}s-C_{k^*,\eps}-F(\theta_0)\frac{\beps(1)}{\aeps(1)}\right)\notag\\
&\quad\quad=\frac{1}{\eps}\int_{k^*}^1\sqrt{F(\ueps(s))-F(\theta_0)}\left(\eps\us(s)-\gamma_{\eps}(s)\right)\sqrt{\frac{\aeps(s)}{2\beps(s)}}\doo\left(\frac{\beps(s)}{\aeps(s)}\right)\,\mbox{d}s+o_{\eps}(1)\\
&\quad\quad=\int_{k^*}^1\sqrt{F(\ueps(s))-F(\theta_0)}\,\us(s)\sqrt{\frac{\aeps(s)}{2\beps(s)}}\doo\left(\frac{\beps(s)}{\aeps(s)}\right)\,\mbox{d}s+o_{\eps}(1).\notag
\end{align}
Here we have used (P1)--(P2), $|\geps(s)|\leq\widetilde{C}\eps^{1/2}$ and the interior estimate (\ref{0211-est3}) to verify
\begin{align*}
\frac{1}{\eps}\left|\int_{k^*}^1\sqrt{F(\ueps(s))-F(\theta_0)}\,\gamma_{\eps}(s)\sqrt{\frac{\aeps(s)}{2\beps(s)}}\doo\left(\frac{\beps(s)}{\aeps(s)}\right)\mbox{d}s\right|\ll1.
\end{align*}
On the other hand, notice that $\sqrt{\frac{\aeps(s)}{2\beps(s)}}\doo\left(\frac{\beps(s)}{\aeps(s)}\right)\in\mbox{C}_{\mathrm{loc}}^{0,\tau}([0,\infty))$. Thus, by (P1) and (P2), we have 
\begin{align*}
\left|\sqrt{\frac{\aeps(s)}{2\beps(s)}}\doo\left(\frac{\beps(s)}{\aeps(s)}\right)-\sqrt{\frac{\aeps(1)}{2\beps(1)}}\doo\left(\frac{\beps(1)}{\aeps(1)}\right)\right|\leq\,C_{11}|s-1|^{\tau},
\end{align*}
 where $C_{11}$ is a positive constant independent of $\eps$. Let us also recall $\us(s)\geq0$ and $\tau\in(0,1)$. Hence, following the similar argument as in (\ref{0214-nigi3}) arrives at the precise leading order term of the expansion in the last line of (\ref{0218-1050}):
\begin{align}\label{0218-1127}
&\int_{k^*}^1\sqrt{F(\ueps(s))-F(\theta_0)}\,\us(s)\sqrt{\frac{\aeps(s)}{2\beps(s)}}\doo\left(\frac{\beps(s)}{\aeps(s)}\right)\,\mbox{d}s\notag\\
&\quad\quad\quad=\int_{1-\eps^{1/2}}^1\sqrt{F(\ueps(s))-F(\theta_0)}\,\us(s)\sqrt{\frac{\aeps(s)}{2\beps(s)}}\doo\left(\frac{\beps(s)}{\aeps(s)}\right)\,\mbox{d}s+o_{\eps}(1)\notag\\
&\quad\quad\quad=\sqrt{\frac{\aeps(1)}{2\beps(1)}}\doo\left(\frac{\beps(1)}{\aeps(1)}\right)\int_{1-\eps^{1/2}}^1\sqrt{F(\ueps(s))-F(\theta_0)}\,\us(s)\,\mbox{d}s+o_{\eps}(1)\\
&\quad\quad\quad=\frac{1}{\aeps(1)}\left(\sqrt{\frac{\aeps(1)}{2\beps(1)}}(\doo\beps)(1)-\sqrt{\frac{\beps(1)}{2\aeps(1)}}(\doo\aeps)(1)\right)\int_{\ueps(1-\eps^{1/2})}^{\ueps(1)}\sqrt{F(t)-F(\theta_0)}\,\mbox{d}t+o_{\eps}(1)\notag\\
&\quad\quad\quad=\frac{1}{\aeps(1)}\left(\sqrt{\frac{\aeps(1)}{2\beps(1)}}(\doo\beps)(1)-\sqrt{\frac{\beps(1)}{2\aeps(1)}}(\doo\aeps)(1)\right)\int_{\theta_0}^{p_0}\sqrt{F(t)-F(\theta_0)}\,\mbox{d}t+o_{\eps}(1).\notag
\end{align}
Since $\frac{\beps(1)}{\aeps(1)}=\mu_0+o_{\eps}(1)$, by (\ref{0218-1050}) and (\ref{0218-1127}), we obtain (\ref{0214-nthu-1}) and complete the proof of Theorem~\ref{lem3}.
\end{proof}

Thanks to Theorem~\ref{lem3}, now we shall establish the precise first two order terms of $\ueps(1)$ and $\us(1)$ with respect to sufficiently small $\eps$. Note that by (\ref{mta-b-0211}) and (\ref{intro-1}), we have
\begin{align}\label{mumu-0219}
\frac{1}{\eps}\left(\frac{\beps(1)}{\aeps(1)}-\mu_0\right)\ll1,\,\,\mbox{as}\,\,0<\eps\ll1.
\end{align}
Combining (\ref{1st-ode-0210-neww}) with (\ref{pohozaev-id})--(\ref{0214-nthu-1}), one may obtain
\begin{align}\label{0219-0754-newadd}
-\frac{(\mte(\ueps(1)))^2}{2}&+\frac{\beps(1)}{\aeps(1)}\left(F(\ueps(1))-F(\theta_0)\right)\notag\\
=&\,\eps\sqrt{\mu_0}\left((N-1)+\frac{\alphas(1)}{\aeps(1)}\right)\left(\int_{\theta_0}^{p_0}\sqrt{2(F(t)-F(\theta_0))}\,\mathrm{d}t+o_{\eps}(1)\right)\notag\\
&+\frac{\eps}{2\aeps(1)}\left(\frac{\betas(1)}{\sqrt{\mu_0}}-\sqrt{\mu_0}\alphas(1)\right)\left(\int_{\theta_0}^{p_0}\sqrt{2(F(t)-F(\theta_0))}\,\mathrm{d}t+o_{\eps}(1)\right)\\
=&\,\eps\sqrt{\mu_0}\left((N-1)+\frac{\alphas(1)}{2\aeps(1)}+\frac{\betas(1)}{2\beps(1)}\right)\left(\int_{\theta_0}^{p_0}\sqrt{2(F(t)-F(\theta_0))}\,\mathrm{d}t+o_{\eps}(1)\right).\notag
\end{align}

The next task at hand is to deal with the first two terms of $\ueps(1)$ and $\us(1)$. By (\ref{0210-ueps-1}), we obtain
\begin{align}\label{eps-219}
\ueps(1)=p+q_{\eps}\,\,\mbox{with}\,\,\lim_{\eps\downarrow0}q_{\eps}=0.
\end{align}
Combining the boundary condition (\ref{bd3}) with (\ref{eps-219}) gives the asymptotics
\begin{align}\label{bd-asy}
\eps\us(1)=\mte(p_0)+q_{\eps}\mte'(p_0)(1+o_{\eps}(1)).
\end{align}
On the other hand, by (\ref{equ-p}) and (\ref{eps-219}), we have, for $0<\eps\ll1$, that
\begin{align}\label{hi-0219}
-\frac{1}{2}(\mte(\ueps(1)))^2&+\frac{\beps(1)}{\aeps(1)}\left(F(\ueps(1))-F(\theta_0)\right)\notag\\
=&\,-\frac{1}{2}\left[\mte(p_0)+q_{\eps}\mte'(p_0)(1+o_{\eps}(1))\right]^2+\frac{\beps(1)}{\aeps(1)}\left[F(p_0)-F(\theta_0)+q_{\eps}f(p_0)(1+o_{\eps}(1))\right]\\
=&\,q_{\eps}\left[-\mte(p_0)\mte'(p_0)+\mu_0f(p_0)+o_{\eps}(1)\right]+\left(\frac{\beps(1)}{\aeps(1)}-\mu_0\right)(F(p_0)-F(\theta_0)).\notag
\end{align}
As a consequence, by (\ref{equ-p}), (\ref{mumu-0219}), (\ref{0219-0754-newadd}) and (\ref{hi-0219}) one may check that
\begin{align}
\frac{q_{\eps}}{\eps}=&({\ds-\mte(p_0)\mte'(p_0)+\mu_0f(p_0)})^{-1}\left[\ds\sqrt{\mu_0}\left((N-1)+\frac{\alphas(1)}{2\aeps(1)}+\frac{\betas(1)}{2\beps(1)}\right)\right.\notag\\
&\hspace*{126pt}\left.\times\int_{\theta_0}^{p_0}\sqrt{2(F(t)-F(\theta_0))}\,\mathrm{d}t+\frac{1}{\eps}\left(\frac{\beps(1)}{\aeps(1)}-\mu_0\right)(F(p_0)-F(\theta_0))\right]+o_{\eps}(1)\notag\\[-0.1em]
&\label{0221-qeps}\\[-0.9em]
=&\left((N-1)+\frac{\alphas(1)}{2\aeps(1)}+\frac{\betas(1)}{2\beps(1)}\right)\frac{\ds\int_{\theta_0}^{p_0}\sqrt{\frac{F(t)-F(\theta_0)}{F(p_0)-F(\theta_0)}}\,\mathrm{d}t}{\ds-\mte'(p_0)+\mu_0\frac{f(p_0)}{\mte(p_0)}}+o_{\eps}(1).\notag
\end{align}
Here we have used (\ref{mumu-0219}) to verify the second equality. By (\ref{eps-219}) and (\ref{0221-qeps}) it yields the precise first two order terms of $\ueps(1)$ with respect to small $\eps$:
\begin{align}\label{hello-0219}
\ueps(1)=p_0+\eps\mathtt{C}_{0}\left((N-1)+\frac{\alphas(1)}{2\aeps(1)}+\frac{\betas(1)}{2\beps(1)}+o_{\eps}(1)\right),
\end{align}
where $\mathtt{C}_{0}=\left(-{\mte'(p_0)}+\mu_0\frac{f(p_0)}{\mte(p_0)}\right)^{-1}{\int_{\theta_0}^{p_0}\sqrt{\frac{F(t)-F(\theta_0)}{F(p_0)-F(\theta_0)}}\,\mathrm{d}t}$ is defined in Theorem~\ref{thm1}. Finally, (\ref{bd-asy}) and (\ref{0221-qeps}) imply
\begin{align}\label{hello-0221}
\us(1)=\frac{\mte(p_0)}{\eps}+\mte'(p_0)\mathtt{C}_{0}\left((N-1)+\frac{\alphas(1)}{2\aeps(1)}+\frac{\betas(1)}{2\beps(1)}\right)+o_{\eps}(1).
\end{align} 

\subsection{Completion of the proofs}\label{sec-thm1}
\noindent


\begin{proof}[$\boldsymbol{\mathrm{Proof\,\,of\,\,Theorem~\ref{thm1}}}$]
 The monotonic increase of $\bigu$ follows immediately from (\ref{ali88}). To deal with the convexness of $\bigu$ as $R\gg1$, let us recall (\ref{eq3}), (P2), (P3) and Lemma~\ref{lem1}. Firstly, we choose $k_{\eps}\in[k^*,1)$ such that $\ueps(k_{\eps})=\frac{\theta_0+p_0}{2}\in(\theta_0,p_0)$. Then by (\ref{0211-est3}) and (\ref{0607-after}) we have $0<\frac{p_0-\theta_0}{2}\leq{L}_{\eps}e^{-\frac{M^*}{\eps}(1-k_{\eps})}$ with $0\stackrel{\eps\downarrow0}{\longleftarrow}\eps\log\frac{p-\theta_0}{2L_{\eps}}\leq{-{M^*}(1-k_{\eps})}<0$, implying
	\begin{align}\label{k-eps-22}
	k^*<1+\frac{\eps}{M^*}\log\frac{p_0-\theta_0}{2L_{\eps}}\leq{k}_{\eps}<1\,\,\mbox{as}\,\,0<\eps\ll1.
	\end{align}
 Moreover, we have $\ds\min_{[k_{\eps},1]}u_{\eps}\geq(\theta_0+p_0)/2$ for any $\eps>0$. Hence, by (\ref{eq3}), (P2), (P3) and (\ref{k-eps-22})	we obtain, for sufficiently small $\eps>0$, that
	\begin{align*}
\eps^2\uss(s)\geq&\,-\eps^2\left(\frac{N-1}{k^*}+\sup_{[k^*,1]}\left|\frac{\alphas}{\aeps}\right|\right)\us(s)+\left(\inf_{[k^*,1]}\frac{\beps(s)}{\aeps(s)}\right)f(\ueps(k^*))\notag\\
\geq&\,-\eps^2\left(\frac{N-1}{k^*}+C_2\right)\us(s)+C_1f(\frac{\theta_0+p}{2})\geq\frac{C_1}{2}f(\frac{\theta_0+p}{2})>0
\end{align*}
since $\ds\lim_{\eps\downarrow0}\sup_{[k^*,1]}\eps\us<\infty$. As a consequence, 
\begin{align*}
\uss(s)>0\,\,\mathrm{in}\,\,[k_{\eps},1]\,\,\mathrm{as}\,\,0<\eps\ll1.
\end{align*}
Along with (\ref{intro-1}) gives $\bigu''>0$ in $[\widetilde{k}_{R},R]$ as $R\gg1$,  where $\widetilde{k}_{R}:=k_{1/R}R\,(=k_{\eps}R)$ admits
\begin{align*}
R+\frac{1}{M^*}\log\frac{p-\theta_0}{2L_{\eps}}\leq\widetilde{k}_{R}<{R}.
\end{align*}  
Hence we obtain the convexness of $\bigu$ near the boundary $r=R$ as $R\gg1$.

It remains to deal with (\ref{id0207-1}). By (\ref{intro-1})--(\ref{new-0208}), (\ref{0209-est2}) and (\ref{0211-est3}), one arrives at
\begin{align*}
0\leq(\frac{r}{R})^{N-1}\bigu'(r)\leq&\frac{2\mte(\theta_0)\mta(R)}{\ds\min_{[0,R]}\mta}e^{-{M^*}(R-r)},\\
0\leq\bigu(r)-\theta_0\leq&\,{L}_{\eps}e^{-{M^*}(R-r)},
\end{align*}
for $r\in[0,R]$. Consequently, we prove (\ref{id0207-1}) with $\mathtt{M}_0=M^*$ and $\ds\mathtt{L}_0={\ds2\mte(\theta_0)\max_{[0,R]}\mta}\left(\ds\min_{[0,R]}\mta\right)^{-1}+\sup_{0<\eps\ll1}{L}_{\eps}$ which are positive constants independent of $R$. Finally, by (\ref{intro-1}), (\ref{new-0208}), (\ref{hello-0219}) and (\ref{hello-0221}), we immediately obtain (\ref{id0207-2}) and (\ref{id0207-3}). Therefore, we complete the proof of Theorem~\ref{thm1}.
\end{proof}

\begin{proof}[$\boldsymbol{\mathrm{Proof\,\,of\,\,Corollary~\ref{rk1}}}$]
Corollary~\ref{rk1}(I) follows directly from (\ref{id0207-2})--(\ref{mathcal-h-0215}) so we omit the proof. We are now in a position to prove Corollary~\ref{rk1}(II). Assume firstly that (i) is satisfied. Setting $\mu_i=\frac{\mtb_i(R)}{\mta_i(R)}$, $i=1,2$, which are independent of $R$, we denote $p_i=p(\mu_{i})$ the unique root of (\ref{equ-p}) with $\mu_0=\mu_i$, $i=1,2$. Notice that $F(p_0)$ is strictly increasing to $p_0\in(\theta_0,\infty)$ and $\mte(p_0)$ is decreasing to $p_0\in(\theta_0,\infty)$ (see (A1) and (A2)). Hence, from (\ref{equ-p}) it is easy to check that $p_0=p_0(\mu_0)$ is strictly decreasing to $\mu_0>0$. As a consequence, the assumption $\mu_1<\mu_2$ implies
\begin{center}
 $p_1>p_2>\theta_0$ and $0<\mte(p_1)<\mte(p_2)$.
\end{center}
 Accordingly, the leading order terms in (\ref{id0207-2}) and (\ref{id0207-3}) immediately imply 
\begin{center}
$\widetilde{\bigu}_{\mta_1,\mtb_1}(R)>\widetilde{\bigu}_{\mta_2,\mtb_2}(R)>\theta_0$ and $0<\widetilde{\bigu}_{\mta_1,\mtb_1}'(R)<\widetilde{\bigu}_{\mta_2,\mtb_2}'(R)$ as $R\gg1$. 
\end{center}

Now we assume that (ii) is satisfied. Then as $R\to\infty$, by (\ref{id0207-2}) we know that $\widetilde{\bigu}_{\mta_1,\mtb_1}(R)$ and $\widetilde{\bigu}_{\mta_2,\mtb_2}(R)$ have the same leading order term, and by (\ref{id0207-3}), $\widetilde{\bigu}_{\mta_1,\mtb_1}'(R)$ and $\widetilde{\bigu}_{\mta_2,\mtb_2}'(R)$ have the same leading order term. Due to the fact that the second and third conditions in (ii) exactly appear in the second order terms of (\ref{id0207-2}) and (\ref{id0207-3}). A simple comparison immediately shows $\widetilde{\bigu}_{\mta_1,\mtb_1}(R)>\widetilde{\bigu}_{\mta_2,\mtb_2}(R)>\theta_0$ and $0<\widetilde{\bigu}_{\mta_1,\mtb_1}'(R)<\widetilde{\bigu}_{\mta_2,\mtb_2}'(R)$ as $R\gg1$. Therefore, we complete the proof of  Corollary~\ref{rk1}(II).
\end{proof}

\noindent

\begin{proof}[$\boldsymbol{\mathrm{Proof\,\,of\,\,Theorem~\ref{thm2}}}$]
It suffices to prove
\begin{align*}
\lim_{R\to\infty}R\int_{k^*R}^{R}(\bigu(r)-\theta_0)\,\mathrm{d}r=\,&\frac{1}{\sqrt{\mu_0}}\int_{\theta_0}^{p_0}\frac{t-\theta_0}{\sqrt{2(F(t)-F(\theta_0))}}\,\mathrm{d}t,\\
\lim_{R\to\infty}R\int_{k^*R}^{R}\bigu'^2(r)\,\mathrm{d}r=\,&\sqrt{\mu_0}\int_{\theta_0}^{p_0}\sqrt{2(F(t)-F(\theta_0))}\,\mathrm{d}t,
\end{align*}
which are equivalent to claiming
\begin{align}
\lim_{\eps\downarrow0}\int_{k^*}^{1}\frac{\ueps(s)-\theta_0}{\eps}\,\mathrm{d}s=\,&\frac{1}{\sqrt{\mu_0}}\int_{\theta_0}^{p_0}\frac{t-\theta_0}{\sqrt{2(F(t)-F(\theta_0))}}\,\mathrm{d}t,\label{g0223-1}\\
\lim_{\eps\downarrow0}\eps\int_{k^*}^{1}\us^2(s)\,\mathrm{d}s=\,&\sqrt{\mu_0}\int_{\theta_0}^{p_0}\sqrt{2(F(t)-F(\theta_0))}\,\mathrm{d}t,\label{g0223-2}
\end{align}
respectively. Firstly, by following the similar argument as the proof of (\ref{pohozaev-id}), (\ref{g0223-2}) can be obtained straightforwardly so we omit the detailed proof. It remains to prove (\ref{g0223-1}). 

To deal with (\ref{g0223-1}), we first consider the decomposition 
\begin{align*}
\int_{k^*}^{1}\frac{\ueps(s)-\theta_0}{\eps}\,\mathrm{d}s=\left\{\int_{k^*}^{1-{\eps}^{\tau_a}}+\int_{1-{\eps}^{\tau_a}}^{1}\right\}\frac{\ueps(s)-\theta_0}{\eps}\,\mathrm{d}s,
\end{align*} 
where $\tau_a\in(0,1)$ has already been used in the proof of Theorem~\ref{lem3}. Due to the interior estimate~(\ref{0211-est3}), we have
\begin{align}\label{hap-end1}
\left|\int_{k^*}^{1-{\eps}^{\tau_a}}\frac{\ueps(s)-\theta_0}{\eps}\,\mbox{d}s\right|\ll1,\,\,\mbox{as}\,\,0<\eps\ll1.
\end{align}
Utilizing (\ref{0212-eq}) and following the similar argument as the proof of (\ref{0214-nthu-1}), we can deal with the second integral as follows. 
\begin{align}\label{hap-end2}
\int_{1-{\eps}^{\tau_a}}^1\frac{\ueps(s)-\theta_0}{\eps}\,\mbox{d}s
=\,&\int_{1-{\eps}^{\tau_a}}^1\frac{\ueps(s)-\theta_0}{\sqrt{\frac{2\beps(s)}{\aeps(s)}\left(F(\ueps(s))-F(\theta_0)\right)}+o_{\eps}(1)}\us(s)\,\mbox{d}s\notag\\
=\,&\frac{1}{\sqrt{\mu_{0}}}\int_{1-{\eps}^{\tau_a}}^1\frac{\ueps(s)-\theta_0}{\sqrt{2\left(F(\ueps(s))-F(\theta_0)\right)}+o_{\eps}(1)}\us(s)\,\mbox{d}s+o_{\eps}(1)\\
=\,&\frac{1}{\sqrt{\mu_{0}}}\int_{\ueps(1-{\eps}^{\tau_a})}^{\ueps(1)}\frac{t-\theta_0}{\sqrt{2\left(F(t)-F(\theta_0)\right)}+o_{\eps}(1)}\mbox{d}t+o_{\eps}(1)\notag\\
=\,&\frac{1}{\sqrt{\mu_{0}}}\int_{\theta_0}^{p_0}\frac{t-\theta_0}{\sqrt{2\left(F(t)-F(\theta_0)\right)}}\mbox{d}t+o_{\eps}(1).\notag
\end{align}
Here we have used (\ref{b-a}), $\ueps(1)\to{p}$, $\ueps(1-{\eps}^{\tau_a})\to\theta_0$ and the fact that $\int_{\theta_0}^{p_0}\frac{t-\theta_0}{\sqrt{2\left(F(t)-F(\theta_0)\right)}}\mbox{d}t$ is finite (cf. Remark~\ref{rk3}) to verify (\ref{hap-end2}). Therefore, (\ref{g0223-1}) follows from (\ref{hap-end1})--(\ref{hap-end2}). The proof of Theorem~\ref{thm2} is done.
\end{proof}

\section{Final remark: How strongly does the small perturbation of $\boldsymbol{\frac{\mtb(R)}{\mta(R)}}$ affect the boundary structure of $\boldsymbol{\bigu}$?}\label{sec-ap}
\noindent

In Theorem~\ref{thm1} we have established refined asymptotics of $\bigu(R)$ and $\bigu'(R)$ under a strong assumption (\ref{mta-b-0211}). The situation shows that, on the boundary asymptotics of $\bigu$, the effect of the perturbation of $\frac{\mtb(R)}{\mta(R)}-\mu_0$ with respect to $R\gg1$ is far smaller than the effect of boundary curvature $\frac{1}{R}$ since 
$\Big|\frac{\mtb(R)}{\mta(R)}-\mu_0\Big|\ll|\pmb{\mathcal{H}}(R)|\sim\frac{1}{R}$ as $R\gg1$.

With regard to the small perturbation of $\frac{\mtb(R)}{\mta(R)}-\mu_0$, particularly for including its significant effect on boundary structure of $\bigu$, we shall pay attention to the situation 
\begin{align}\label{0616-11pm}
\lim_{R\to\infty}\frac{\mtb(R)}{\mta(R)}=\mu_0\,\,\mathrm{and}\,\,\liminf_{R\to\infty}R\left|\frac{\mtb(R)}{\mta(R)}-\mu_0\right|>0.
\end{align}
The main difference between (\ref{mta-b-0211}) and \eqref{0616-11pm} comes from the fact that \eqref{0616-11pm} implies
\begin{align}\label{en-0621}
\left|\pmb{\mathcal{H}}(R)\right|\lesssim\left|\frac{\mtb(R)}{\mta(R)}-\mu_0\right|\,\,\mathrm{as}\,\,R\gg1.
\end{align}
Accordingly, the perturbation of $\frac{\mtb(R)}{\mta(R)}$ around $\mu_0$ plays a crucial role in asymptotic behaviors of $\bigu(R)$ and $\bigu'(R)$ and is undoubtedly not to be ignored.
Note also that  \eqref{0616-11pm} includes \eqref{0608-hap}. Hence, (\ref{mumu-0219}) is no longer satisfied, and the asymptotic expansions of $\bigu(R)$ and $\bigu'(R)$ are more complicated than the corresponding results in Theorem~\ref{thm1}. Such a result is stated as follows. 

\begin{corollary}\label{cor0603}
Under the hypotheses as in Theorem~\ref{thm1}, we replace (\ref{mta-b-0211}) with \eqref{0616-11pm}. Then as $R\gg1$, we have 
\begin{align}
\bigu(R)=&\,\quad{p_0}+\frac{\sqrt{F(p_0)-F(\theta_0)}}{\mu_0f(p_0)-\mte(p_0)\mte'(p_0)}\left(\frac{\mtb(R)}{\mta(R)}-\mu_0\right)+\boldsymbol{\mathtt{C}_{0}}\pmb{\mathcal{H}}(R)+\frac{o(1)}{R},\label{id0304-2}\\
\bigu'(R)=&\,\mte(p_0)+\mte'(p_0)\left(\frac{\sqrt{F(p_0)-F(\theta_0)}}{\mu_0f(p_0)-\mte(p_0)\mte'(p_0)}\left(\frac{\mtb(R)}{\mta(R)}-\mu_0\right)+\boldsymbol{\mathtt{C}_{0}}\pmb{\mathcal{H}}(R)\right)+\frac{o(1)}{R}.\label{id0304-3}
\end{align}
\end{corollary}
\begin{proof}
The argument is similar to \eqref{0221-qeps}--\eqref{hello-0221}, where we should note that the second equality of (\ref{0221-qeps}) is obtained from (\ref{mumu-0219}) (which is equivalent to (\ref{mta-b-0211})). Note also that (\ref{equ-p}) and the first equality of (\ref{0221-qeps}) still hold under assumption \eqref{0616-11pm}. Since \eqref{0616-11pm} cannot imply (\ref{mumu-0219}), we shall use the first equality of (\ref{0221-qeps}) and \eqref{equ-p} to obtain that, as $\eps=\frac{1}{R}\to\infty$,
\begin{align}\label{0304-qeps}
q_{\eps}
=&\frac{1}{R}\left(\ds-\mte'(p_0)+\mu_0\frac{f(p_0)}{\mte(p_0)}\right)^{-1}\left[\left((N-1)+\frac{\alphas(1)}{2\aeps(1)}+\frac{\betas(1)}{2\beps(1)}\right)\right.\notag\\
&\hspace*{128pt}\left.\times\int_{\theta_0}^{p_0}\sqrt{\frac{F(t)-F(\theta_0)}{F(p_0)-F(\theta_0)}}\,\mathrm{d}t+\frac{1}{\sqrt{2\mu_0}\eps}\left(\frac{\beps(1)}{\aeps(1)}-\mu_0\right)+o_{\eps}(1)\right]\\
=&\,\frac{\sqrt{F(p_0)-F(\theta_0)}}{\mu_0f(p_0)-\mte(p_0)\mte'(p_0)}\left(\frac{\mtb(R)}{\mta(R)}-\mu_0\right)+\boldsymbol{\mathtt{C}_{0}}\pmb{\mathcal{H}}(R)+\frac{o(1)}{R}.\notag
\end{align}
As a consequence, by (\ref{intro-1})--(\ref{new-0208}), (\ref{hello-0219})--(\ref{hello-0221}) and (\ref{0304-qeps}), we get (\ref{id0304-2}) and (\ref{id0304-3}) and end the proof of Corollary~\ref{cor0603}.
\end{proof}

At the end of this note, we take a holistic viewpoint to answer the question on the title of this section.

 \begin{remark}\label{rk0621}
To see the effect of $\frac{\mtb(R)}{\mta(R)}-\mu_0$ on asymptotics of $\bigu$, we may assume $\frac{\mtb(R)}{\mta(R)}-\mu_0=\mu_*R^{-\tau_*}$ with $\mu_*\neq0$ and $\tau_*>0$. We stress that different $\tau_*$ results in the various asymptotics of $\bigu$. More precisely, by \eqref{id0304-2}--\eqref{id0304-3} we have 
\begin{center}
$\ds|\bigu(R)-{p_0}|+|\bigu'(R)-\mte(p_0)|\lesssim{R^{-\min\{1,\tau_*\}}}$ as $R\gg1$. 
\end{center}
Moreover,
\begin{itemize}
\item If $0<\tau_*<1$, then $\left|\frac{\mtb(R)}{\mta(R)}-\mu_0\right|\gg|\pmb{\mathcal{H}}(R)|$ and
\begin{align*}
\bigu(R)=&\,\quad{p_0}+\frac{\mu_*\sqrt{F(p_0)-F(\theta_0)}}{\mu_0f(p_0)-\mte(p_0)\mte'(p_0)}\frac{1}{R^{\tau_*}}+\frac{o(1)}{R^{\tau_*}},\\[0.3mm]
\bigu'(R)=&\,\mte(p_0)+\frac{\mu_*\mte'(p_0)\sqrt{F(p_0)-F(\theta_0)}}{\mu_0f(p_0)-\mte(p_0)\mte'(p_0)}\frac{1}{R^{\tau_*}}+\frac{o(1)}{R^{\tau_*}}.
\end{align*}
Note also that if $\mu_*<0$ (resp., $>0$), there holds $\bigu(R)<p_0$ (resp., $>p_0$) and $\bigu'(R)>\mte(p_0)$ (resp., $<\mte(p_0)$) as $R\gg1$. 
\item If $\tau_*=1$, then $\left|\frac{\mtb(R)}{\mta(R)}-\mu_0\right|\sim|\pmb{\mathcal{H}}(R)|$ and
\begin{align*}
\bigu(R)=&\,\quad{p_0}+\left(\frac{\mu_*\sqrt{F(p_0)-F(\theta_0)}}{\mu_0f(p_0)-\mte(p_0)\mte'(p_0)}\frac{1}{R}+\boldsymbol{\mathtt{C}_{0}}\pmb{\mathcal{H}}(R)\right)+\frac{o(1)}{R},\\[0.2mm]
\bigu'(R)=&\,\mte(p_0)+\mte'(p_0)\left(\frac{\mu_*\sqrt{F(p_0)-F(\theta_0)}}{\mu_0f(p_0)-\mte(p_0)\mte'(p_0)}\frac{1}{R}+\boldsymbol{\mathtt{C}_{0}}\pmb{\mathcal{H}}(R)\right)+\frac{o(1)}{R}.
\end{align*}
\item If $\tau_*>1$, then $\left|\frac{\mtb(R)}{\mta(R)}-\mu_0\right|\ll|\pmb{\mathcal{H}}(R)|$ and
\begin{align*}
\bigu(R)=&\,\quad{p_0}+\boldsymbol{\mathtt{C}_{0}}\pmb{\mathcal{H}}(R)+\frac{o(1)}{R},\\[0.1mm]
\bigu'(R)=&\,\mte(p_0)+\mte'(p_0)\boldsymbol{\mathtt{C}_{0}}\pmb{\mathcal{H}}(R)+\frac{o(1)}{R}.
\end{align*}
\end{itemize}
\end{remark}


\end{document}